\documentclass[11]{amsart}

\usepackage{etex}
\usepackage{amsmath, amssymb}
\usepackage[frame,cmtip,arrow,matrix,line,graph,curve]{xy}
\usepackage{graphpap, color, paralist, pstricks}
\usepackage[mathscr]{eucal}
\usepackage[pdftex,colorlinks,backref=page,citecolor=blue]{hyperref}
\usepackage{tikz}

\newtheorem{theorem}{Theorem}[section]
\newtheorem{proposition}[theorem]{Proposition}
\newtheorem{corollary}[theorem]{Corollary}
\newtheorem{lemma}[theorem]{Lemma}

\theoremstyle{definition}
\newtheorem{definition}[theorem]{Definition}

\newtheorem{question}[theorem]{Question}
\newtheorem{remark}[theorem]{Remark}

\newtheorem*{acknowledgement}{Acknowledgement}

\newcommand{\PP}{\mathbb{P}}
\newcommand{\QQ}{\mathbb{Q}}
\newcommand{\CC}{\mathbb{C}}
\newcommand{\RR}{\mathbb{R}}
\newcommand{\ZZ}{\mathbb{Z}}
\newcommand{\NN}{\mathbb{N}}

\newcommand{\cE}{\mathcal{E} }

\newcommand{\cM}{\mathcal{M} }

\newcommand{\cO}{\mathcal{O} }
\newcommand{\cP}{\mathcal{P} }
\newcommand{\cQ}{\mathcal{Q} }

\newcommand{\cV}{\mathcal{V} }

\newcommand{\proj}{\mathrm{Proj}\;}
\newcommand{\id}{\mathrm{id}}
\newcommand{\im}{\mathrm{im}}
\newcommand{\Hom}{\mathrm{Hom}}
\newcommand{\cHom}{\mathcal{H}om}
\newcommand{\pdeg}{\mathrm{pardeg}\;}
\newcommand{\cpHom}{\mathcal{P}ar\mathcal{H}om}
\newcommand{\cpEnd}{\mathcal{P}ar\mathcal{E}nd}
\newcommand{\cspHom}{\mathcal{SP}ar\mathcal{H}om}
\newcommand{\pHom}{\mathrm{ParHom}}
\newcommand{\pEnd}{\mathrm{ParEnd}}
\newcommand{\spHom}{\mathrm{SParHom}}

\newcommand{\Ext}{\mathbf{Ext}}

\newcommand{\rank}{\mathrm{rank}\, }

\def\SL{\mathrm{SL}}
\def\git{/\!/ }

\def\Pic{\mathrm{Pic}}
\def\im{\mathrm{im}}

\begin{document}

\title[Birational geometry of the moduli of parabolic bundles]{Birational geometry of the moduli space of rank 2 parabolic vector bundles on a rational curve}

\author{Han-Bom Moon}
\address{Department of Mathematics, Fordham University, Bronx, NY 10458, USA}
\email{hmoon8@fordham.edu}

\author{Sang-Bum Yoo}
\address{Department of Mathematics, POSTECH, Pohang, Gyungbuk 790-784, Republic of Korea}
\email{sangbumyoo@postech.ac.kr}

\begin{abstract}
We investigate the birational geometry (in the sense of Mori's program) of the moduli space of rank 2 semistable parabolic vector bundles on a rational curve. We compute the effective cone of the moduli space and show that all birational models obtained by Mori's program are also moduli spaces of parabolic vector bundles with certain parabolic weights.
\end{abstract}

\maketitle

\section{Introduction}

In the last decade, it has been proved that studying the geometry of a moduli space in the viewpoint of the minimal model program (or \textbf{Mori's program}) for a moduli space is very fruitful. Mori's program for a moduli space $M$ consists of the following steps: 1) Compute the cone of effective divisors $\mathrm{Eff}(M)$. 2) For each divisor $D \in \mathrm{Eff}(M)$, find the projective model 
\[
	M(D) := \proj \bigoplus_{m \ge 0}H^{0}(M, \cO(\lfloor mD\rfloor)).
\]
3) Study the moduli theoretic interpretation (if there is) of $M(D)$ and its relation with $M$.

There are several intensively studied examples. For the moduli space $\overline{\cM}_{g}$ of stable curves, the famous Hassett-Keel program is a study of birational models of the form $\overline{\cM}_{g}(K_{\overline{\cM}_{g}}+\alpha D)$ with the boundary $D$ of singular curves and $\alpha \le 1$. It has been shown that many of these models are indeed moduli spaces of curves with worse singularities (for a nice overview, see \cite{FS13a}). For Hilbert scheme $\mathrm{Hilb}^{n}(\PP^{2})$ of $n$ points on $\PP^{2}$, many of its birational models appearing in Mori's program are moduli spaces of Bridgeland stable objects in $D^{b}(\PP^{2})$ with certain stability condition (\cite{ABCH13}). For the moduli space of stable sheaves $\mathrm{M}_{H}(v)$ on a K3 surface $X$, all flips of $\mathrm{M}_{H}(v)$ are moduli spaces of Bridgeland stable objects in $D^{b}(X)$ (\cite{BM13}).

\subsection{The main result of the paper} 

The aim of this paper is to investigate the birational geometry of the moduli space $\cM(\vec{a})$ of rank 2 semistable parabolic vector bundles of degree 0 on $\PP^{1}$, in the sense of Mori's program. The moduli functor depends on a parabolic weight $\vec{a}$, which imposes a certain stability condition. If we vary $\vec{a}$, then the moduli space changes. The study of this change has been well understood by many authors in \cite{Bau91, Ber94, BH95, Tha96, Tha02}. All birational morphisms between them are able to be described in terms of smooth blow-ups/downs, or variation of GIT. In this paper we revisit these birational modifications in terms of Mori's program. 

The following is the first main result of this paper, which is the first step of Mori's program. Let $n \ge 5$ be the number of parabolic points. 

\begin{theorem}[Theorem \ref{thm:effectiveconegeneral}]\label{thm:effectiveconegeneralintro}
Let $\vec{a}$ be a general parabolic weight such that $\cM(\vec{a})$ has the maximal Picard number $n+1$. Then the effective cone $\mathrm{Eff}(\cM(\vec{a}))$ is polyhedral and there are precisely $2^{n-1}$ extremal rays. 
\end{theorem}

Note that the computation of the effective cone of a variety is a hard problem in general. Except toric varieties, there are few examples of varieties with large Picard number and known effective cone. Among moduli spaces, most of examples with known effective cone have Picard number $\le 2$ or have a simplicial effective cone (for example, the moduli space of $n$-unordered pointed rational curves $\overline{\mathrm{M}}_{0,n}/S_{n}$ (\cite{KM13}), the moduli space of stable maps $\overline{\mathrm{M}}_{0,0}(\PP^{d}, d)$ (\cite{CHS08})). Theorem \ref{thm:effectiveconegeneralintro} provides one highly nontrivial example of an algebraic variety with completely known non-simplicial effective cone. 

After the computation of $\mathrm{Eff}(\cM(\vec{a}))$, the following theorem  is a simple consequence of the work of Pauly on generalized theta divisors (\cite{Pau96}). 

\begin{theorem}[Theorem \ref{thm:Moriprogram}]\label{thm:Moriprogramintro}
For any divisor $D \in \mathrm{int}\;\mathrm{Eff}(\cM(\vec{a}))$, the birational model $\cM(\vec{a})(D)$ is isomorphic to $\cM(\vec{b})$ for some parabolic weight $\vec{b}$. 
\end{theorem}

Indeed, even in the case that $D \in \partial \mathrm{Eff}(\cM(\vec{a}))$, we can describe the projective models as moduli spaces of parabolic bundles with fewer parabolic points (Remark \ref{rem:boundary}). In short, \textbf{all} projective models of $\cM(\vec{a})$ appearing in Mori's program of $\cM(\vec{a})$ are moduli spaces of parabolic vector bundles with certain degree and stability condition. 

Therefore as opposed to the case of Hilbert schemes and moduli spaces of ordinary stable sheaves, there are no newly appearing moduli spaces parametrizing objects in (some) derived categories. In this sense, the (non separated, non connected) moduli stack of rank 2 parabolic bundles over $\PP^{1}$ is closed under Mori's program. 

We may ask a generalized question. It would be interesting if one can show a similar statement for moduli spaces of an arbitrary rank parabolic bundles over an arbitrary genus $g$ smooth curve. 

\begin{question}
Let $C$ be a smooth projective curve. Let $\cM_{C}(r, \vec{a}, L)$ be the moduli space of rank $r$ semistable parabolic bundles with fixed determinant $L$ over $C$. Are all birational models of $\cM_{C}(r, \vec{a}, L)$ appearing in Mori's program of the form $\cM_{C}(r', \vec{b}, L')$?
\end{question}

After finishing the first draft of this paper, we recognized that the Cox ring of the moduli space of parabolic bundles has been studied by several mathematicians. In \cite{Muk05}, Mukai used the moduli space of rank two, degree one parabolic bundles on $\PP^{1}$ with a special parabolic weight to study the finite generation of certain invariant ring. Recently, Manon showed that the Cox ring of the moduli stack of rank two parabolic bundles with generic parabolic points is generated by level one and two conformal blocks (\cite{Man09b}). 

\subsection{Outline of the proof}

Our approach in attacking this problem was to find an elementary construction of $\cM(\vec{a})$. We show that when $\vec{a}$ is very small, then $\cM(\vec{a}) \cong (\PP^{1})^{n}\git_{L}\SL_{2}$, where $L = \cO(a_{1}, \cdots, a_{n})$ (Proposition \ref{prop:parabolicGIT}). When $\vec{a}$ becomes large, in Proposition \ref{prop:generalmoduli} we show that $\cM(\vec{a})$ is (possibly a flip of) the blow-up of $(\PP^{1})^{n}\git_{L}\SL_{2}$ at a smooth point.

The computation of $\mathrm{Eff}((\PP^{1})^{n}\git_{L}\SL_{2})$ is a classical result in invariant theory. We restate the result in Proposition \ref{prop:effectiveconeGIT}. When the Picard number is maximal, $\mathrm{Eff}((\PP^{1})^{n}\git_{L}\SL_{2})$ is a cone over the hypersimplex $\Delta(2, n)$. 

Now to compute $\mathrm{Eff}(\cM(\vec{a}))$ for a general $\vec{a}$, it suffices to compute the effective cone of a single point blow-up of $(\PP^{1})^{n}\git_{L}\SL_{2}$. To do so, we use the combinatorics of $\mathfrak{sl}_{2}$-conformal blocks. Originally, the conformal block was defined by using representations of affine Lie algebra, but it is well-known that there is an interesting connection with moduli spaces of parabolic vector bundles (\cite{Pau96}). Furthermore, in the $\mathfrak{sl}_{2}$-case, its rich algebraic/combinatorial structure is well understood by the work of many mathematicians, including Looijenga, Swinarski and B. Alexeev. After introducing two equivalent combinatorial models to count $\mathfrak{sl}_{2}$-conformal blocks, we show that any effective divisor on $\cM(\vec{a})$ is a nonnegative linear combination of $2^{n-1}$ level one conformal blocks. Thus we obtain Theorem \ref{thm:effectiveconegeneralintro}.

Finally, in \cite{Pau96}, Pauly introduced a natural ample divisor on $\cM(\vec{a})$ and showed that it is isomorphic to conformal blocks with certain weights. By using this result, we finish the proof of Theorem \ref{thm:Moriprogramintro}. 

\subsection{Structure of the paper}

The organization of this paper is as follows. In Section \ref{sec:prelimparbundle} we review the definition and basic properties of moduli spaces of semistable rank 2 parabolic vector bundles. Also we state some known results on their wall-crossing behavior. In Section \ref{sec:GIT}, we give an elementary construction of $\cM(\vec{a})$ as a simple GIT quotient. In Section \ref{sec:effectiveconeGIT}, we compute the Picard group and the effective cone of the GIT quotient that appeared in the previous section. Section \ref{sec:conformalblock} reviews an elementary definition and combinatorics of $\mathfrak{sl}_{2}$-conformal blocks. In Section \ref{sec:effectiveconegeneral}, we compute $\mathrm{Eff}(\cM(\vec{a}))$ for an effective parabolic weight $\vec{a}$. Finally in Section \ref{sec:thetadivisor}, we prove Theorem \ref{thm:Moriprogramintro}.

\subsection*{Notation and conventions}

We work over an algebraically closed field $\CC$ of characteristic 0. In this paper, we fix $n \ge 3$ distinct parabolic points $\vec{p} = (p_{1}, \cdots, p_{n})$ on $\PP^{1}$. The notion and combinatorics of parabolic bundles are significantly simplified for the rank 2 case. So in this paper, our discussions are focused on the rank 2 case only. We denote the set $\{1, \cdots, n\}$ by $[n]$.

\section{Preliminaries on the moduli space of parabolic vector bundles}\label{sec:prelimparbundle}

In this section, we review some of basics on parabolic vector bundles on $\PP^{1}$ and the moduli spaces of them. After that we review some known results on birational geometry of the moduli spaces. For details and proofs of the results in this section, see \cite{BH95, MS80, Tha96, Tha02, Yok95}.

\subsection{Moduli space of parabolic vector bundles}
A \textbf{rank 2 parabolic vector bundle on $\PP^{1}$} with parabolic structure at $\vec{p}$ is a collection $(E, \{V_{i}\}, \vec{a})$ where
\begin{enumerate}
\item $E$ is a vector bundle of rank 2 over $\PP^{1}$;
\item for each $i \in [n]$, $V_{i} \subset E|_{p_{i}}$ is a 1-dimensional subspace;
\item $\vec{a} = (a_{1}, \cdots, a_{n})$ is a sequence of rational numbers such that $0 \le a_{i}  < 1$, called a parabolic weight.
\end{enumerate}

Sometimes we write $(E, \{V_{i}\}, \vec{a})$ simply as $(E, \{V_{i}\})$, $(E, \vec{a})$ or even $E$, if there is no confusion. The set of all possible parabolic weights is $\overline{W} := ([0, 1) \cap \QQ)^{n}$. The interior of $\overline{W}$, which parametrizes positive parabolic weights, is denoted by $W$.

If we consider the moduli stack of parabolic vector bundles, it is highly non-separated even if we fix the rank and the degree of the underlying vector bundle. The notion of (semi)stability of parabolic vector bundles enables us to obtain a proper open substack.

A \textbf{parabolic line bundle} $(E, \vec{b})$ is simply a pair of line bundle $E$ and a parabolic weight $\vec{b} = (b_{1}, \cdots, b_{n})$. Let $(E, \{V_{i}\}, \vec{a})$ be a rank 2 parabolic vector bundle on $\PP^1$. A \textbf{parabolic subbundle} $(E', \vec{b})$ is a parabolic line bundle where $E' \subset E$ is a subbundle and 
\[
	b_{i} = \begin{cases} a_{i}, & E'|_{p_{i}} = V_{i}\\
	0, & E'|_{p_{i}}\ne V_{i}.\end{cases}
\]
A \textbf{parabolic quotient bundle} $(E'', \vec{c})$ is a parabolic line bundle where $E''$ is a quotient bundle and if $q : E \to E''$ is the quotient map, 
\[
	c_{i} = \begin{cases} a_{i}, & q|_{p_i}(V_{i}) \ne 0\\
	0, & q|_{p_i}(V_{i}) = 0.\end{cases}
\]

For a rank 1 or 2 parabolic bundle $(E, \vec{a})$, the \textbf{parabolic degree} of $E$ is 
\[
    \pdeg E= \deg E + \sum_{i=1}^{n}a_{i}.
\]
Finally, for a parabolic bundle $E$, the \textbf{parabolic slope} of $E$ is defined as
\[
	\mu(E) = \frac{\pdeg E}{\rank E}.
\]

\begin{definition}
A rank 2 parabolic bundle $(E, \{V_{i}\}, \vec{a})$ is \textbf{(semi)stable} if for every parabolic subbundle $(E',\vec{b})$,
\[
	\mu(E') \; (\le) < \mu(E).
\]
\end{definition}

We say that two rank 2 semistable parabolic vector bundles are \textbf{S-equivalent} if they have the same factors in their Jordan-H\"older filtrations. In concrete terms, this equivalence relation is generated by the following: If $(E, \{V_{i}\}, \vec{a})$ is semistable and $(E', \vec{b}) \hookrightarrow (E, \{V_{i}\}, \vec{a})$ is a parabolic subbundle such that $\mu(E') = \mu(E)$, then $E \equiv E' \oplus E/E'$. By definition, if $E$ is stable, then it is S-equivalent to itself only. 

Let $\cM(\vec{a},d)$ be the coarse moduli space of S-equivalent classes of rank 2, degree $d$, semistable parabolic vector bundles on $\PP^{1}$ with parabolic structure $\vec{a}$ at $\vec{p}$. We denote $\cM(\vec{a}, 0)$ by $\cM(\vec{a})$. We denote the open subvariety of $\cM(\vec{a})$ parametrizing stable parabolic vector bundles by $\cM(\vec{a}, d)^{s}$ (or $\cM(\vec{a})^{s}$ if $d = 0$). 

\begin{theorem}[\protect{\cite[Theorem 4.1]{MS80}}]\label{moduli par bdls}
For $\vec{a} \in W$, the moduli space $\cM(\vec{a},d)$ is an irreducible normal projective variety of dimension $n-3$, if it is nonempty.
\end{theorem}

\subsection{Deformation theory of parabolic vector bundles}\label{sec:deformation}

The deformation theory of parabolic vector bundles has been worked out by Yokogawa in \cite{Yok95} in great generality. 

Let $(E, \{V_{i}\}, \vec{a})$ and $(F, \{W_{i}\}, \vec{b})$ be two rank 2 parabolic vector bundles. A bundle morphism $f : E \to F$ is called \textbf{(strongly) parabolic} if $f(V_{i}) = 0$ whenever $a_{i}\; (\ge) > b_{i}$. We shall denote by $\cpHom(E,E')$ and $\cspHom(E,E')$ the sheaves of parabolic and strongly parabolic morphisms, and by $\pHom(E,E')$ and $\spHom(E,E')$ their global sections respectively. We also use the notation $\cpEnd(E):=\cpHom(E,E)$ and $\pEnd(E):=\pHom(E,E)$. 

The following fact is a standard consequence of the notion of the stability, as that of ordinary vector bundles. The proof is identical to that of \cite[Proposition 4.7, Corollary 4.8]{Fri98}.

\begin{proposition}\label{mor of stable par bdls}
Let $E$ and $F$ be stable parabolic bundles such that $\mu(E) \ge \mu(F)$. Then $\dim \pHom(E,E')=1$ if $E$ and $F$ are isomorphic, and $0$ otherwise. In particular, $\pEnd(E)=\CC\cdot\id.$
\end{proposition}

The category of parabolic bundles on $\PP^1$ is not abelian. However, Yokogawa showed that it is contained in an abelian category $\cP$ as a full subcategory using a generalized notion of parabolic sheaves, and $\cP$ has enough injective objects. For each parabolic vector bundle $E$, $\Ext^{i}(E,-)$ is defined by the $i$-th right derived functor of $\pHom(E,-)$ in $\cP$.

\begin{lemma}[\protect{\cite[Theorem 1.4, 3.6]{Yok95}}]\label{lem:Extordinarycohomology}
Let $E_{1}, E_{2}$ be two parabolic bundles. Then 
\[
	\Ext^{i}(E_{2},E_{1})\cong \mathrm{H}^{i}(\cpHom(E_{2},E_{1}))
\] 
for $i \ge 0$.
\end{lemma}

For each parabolic weight $\vec{a}$, let $I_{\vec{a}} = \{i \in [n]\;|\; a_{i} \ne 0\}$. We say that two weights $\vec{b}$ and $\vec{c}$ are \textbf{complementary} if $I_{\vec{b}} \sqcup I_{\vec{c}}$ defines a partition of $\{1, \cdots, n\}$. When two parabolic line bundles $(E, \vec{b})$ and $(F, \vec{c})$ have complementary weights, an \textbf{extension} of $(F,\vec{c})$ by $(E,\vec{b})$ is a short exact sequence of parabolic morphisms
\[
	\xymatrix{0\ar[r]&(E,\vec{b})\ar[r]&(G,\{V_{i}\},\vec{a})
	\ar[r]^-{q}&(F,\vec{c})\ar[r]&0}
\]
where 
\[
	a_{i}=\begin{cases}b_{i}, & i \in I_{\vec{b}}\\ c_{i},& i \in I_{\vec{c}}\end{cases}
\]
and $V_{i} = E|_{p_{i}}$ if $i \in I_{\vec{b}}$ and $q(V_{i}) \ne 0$ if $i \in I_{\vec{c}}$. It is obvious that $(E, \vec{b})$ (resp. $(F, \vec{c})$) is a parabolic subbundle (resp. quotient bundle) of $(G, \{V_{i}\}, \vec{a})$.

The proof of the following proposition is identical to the arguments in \cite[p. 31]{Fri98}.

\begin{proposition}\label{ext of par line bdls}
For two parabolic line bundles $(E, \vec{b})$ and $(F, \vec{c})$ which have complementary weights, the isomorphism classes, up to scalar multiplications, of nonsplit extensions of $F$ by $E$ are parametrized by $\PP\Ext^{1}(F, E)$.
\end{proposition}

We have a generalized Serre duality for parabolic bundles.

\begin{proposition}[\protect{\cite[Proposition 3.7]{Yok95}}]\label{Serre duality}
For any parabolic bundles $E$ and $F$, there are natural isomorphisms
\[
	\Ext^{1-i}(E,F\otimes\cO_{\PP^1}(n-2)) \cong 
	\mathrm{H}^i(\cspHom(F,E))^{\vee}
\]
for $i = 0, 1$.
\end{proposition}

\begin{theorem}[\protect{\cite[Theorem 2.4]{Yok95}}]\label{def of par bdl}
Let $(E, \{V_{i}\}, \vec{a})$ be a rank 2 parabolic bundle corresponding to a geometric point $x$ of $\cM(\vec{a})^{s}$. The Zariski tangent space of $\cM(\vec{a})^{s}$ at $x$ is naturally isomorphic to $\Ext^{1}(E, E)$. If $\Ext^{2}(E, E) = 0$, then $\cM(\vec{a})^{s}$ is smooth at $x$.
\end{theorem}

\begin{corollary}\label{smoothness of moduli}
$\cM(\vec{a})^s$ is smooth. 
\end{corollary}
\begin{proof}
By Lemma \ref{lem:Extordinarycohomology}, $\Ext^{2}(E, E) \cong \mathrm{H}^{2}(\cpHom(E, E))$. The latter cohomology is zero since it is an ordinary sheaf cohomology on a curve.
\end{proof}

\subsection{Wall crossing}\label{sec:wallcrossing}

We devote this subsection to showing how $\cM(\vec{a})$ changes when $\vec{a}$ varies. The birational map between $\cM(\vec{a})$ and $\cM(\vec{a}')$ with two adjacent parabolic weights $\vec{a}$ and $\vec{a}'$ is studied in \cite{BH95} and \cite[Section 7]{Tha96}. 

\begin{remark}
In \cite{BH95, Tha96}, the authors stated the result in the case that there is only one parabolic point. But the result is generalized to the case of an arbitrary number of parabolic points in a straightforward way. 
\end{remark}

\begin{definition}\label{def:effecivegeneralweights}
A parabolic weight $\vec{a} \in W$ is called \textbf{effective} if $\cM(\vec{a})^{s} \ne \emptyset$. An effective weight is called \textbf{general} if $\cM(\vec{a}) = \cM(\vec{a})^{s}$. 
\end{definition}

By Corollary \ref{smoothness of moduli}, for a general parabolic weight, $\cM(\vec{a})$ is smooth.

Let us study stability walls on $W$. Let $(E = \cO(k)\oplus \cO(-k),\{V_{i}\}, \vec{a})$ be a parabolic vector bundle over $\PP^{1}$ for some nonnegative $k$. If it is strictly semistable (hence it is on a wall), then there is a parabolic subbundle $(F = \cO(-m), \vec{b})$ such that $\mu(F) = \mu(E)$. Let $I = \{i \in \{1, \cdots,n\}\;|\;F|_{p_{i}} = V_{i}\}$. Then
\[
	-m + \sum_{i \in I}a_{i} = \mu(F) = \mu(E)
	= \frac{1}{2}\sum_{i=1}^{n}a_{i},
\]
so 
\[
	\mu(F) - \mu(E) = \sum_{i \in I}a_{i} - \sum_{i \in I^{c}}a_{i} = 2m.
\]
Therefore all stability walls are defined by
\begin{equation}\label{eqn:stabilitywall}
	\Delta_{I, m} = \{(a_{1}, \cdots, a_{n}) \in W\;|\;
	\sum_{i \in I}a_{i} - \sum_{i \in I^{c}}a_{i} = 2m\}.
\end{equation}

\begin{lemma}\label{lem:stabilitywall}
The space of positive parabolic weights $W$ is decomposed into finitely many chambers by walls $\Delta_{I, m}$ for $I \subset \{1, \cdots, n\}$ and $m \in \ZZ$.
\end{lemma}

Note that $\Delta_{I, m} = \Delta_{I^{c}, -m}$. 

Next, we see what parabolic bundles become unstable as we cross a stability wall. It suffices to analyze the change under a simple wall-crossing along the relative interior of a wall. Choose a general point $\vec{a}$ in $\Delta_{I, m}$. Let $\Delta_{I,m}^+$ and $\Delta_{I,m}^-$ be small neighborhoods of at $\vec{a}$ in
\[
	\{(b_{1},\cdots,b_{n})\in W\;|\;
	\sum_{i\in I}b_{i}-\sum_{i\in I^c}b_{i}>2m\}
	\quad \mathrm{ and }\quad
	\{(b_{1},\cdots,b_{n})\in W\;|\;
	\sum_{i\in I}b_{i}-\sum_{i\in I^c}b_{i}<2m\}
\]
respectively. The stability coincide with the semistability on $\Delta_{I,m}^{\pm}$. A parabolic bundle is \textbf{$\Delta_{I,m}^{+}$-stable} (resp. \textbf{$\Delta_{I,m}^{-}$-stable}) if it is stable with respect to parabolic weights in $\Delta_{I,m}^{+}$ (resp. $\Delta_{I,m}^{-}$). We look for parabolic bundles which are $\Delta_{I,m}^-$-stable but $\Delta_{I,m}^+$-unstable. 

\begin{lemma}\label{max dest}
If $(\cO(k)\oplus\cO(-k),\{V_i\})$ is $\Delta_{I,m}^-$-stable but $\Delta_{I,m}^+$-unstable, then any destabilizing subbundle is of the form $(\cO(-m),\vec{b})$ and $I=I_{\vec{b}} := \{i \in[n]\;|\;\cO(-m)|_{p_i}=V_{i}\}$. 
\end{lemma}

\begin{proof}
Since $(\cO(k)\oplus\cO(-k),\{V_i\})$ is $\Delta_{I,m}^+$-unstable, we have a destabilizing subbundle $\cO(-m')$ of $(\cO(k)\oplus\cO(-k),\{V_i\})$ such that 
\[
	-m' +  \sum_{i\in I_{\vec{b}}}a_{i} > \frac{1}{2}\sum_{i=1}^{n}a_{i}
\]
for any $\vec{a}\in\Delta_{I,m}^+$. Since $(\cO(k)\oplus\cO(-k),\{V_i\})$ is $\Delta_{I,m}^-$-stable,
\[
	-m' +  \sum_{i\in I_{\vec{b}}}a_{i} < \frac{1}{2}\sum_{i=1}^{n}a_{i}
\]
for any $\vec{a}\in\Delta_{I,m}^-$. Then $\Delta_{I,m}^-\subset\Delta_{I_{\vec{b}},m'}^-$ and $\Delta_{I,m}^+\subset\Delta_{I_{\vec{b}},m'}^+$. Hence $\Delta_{I,m}=\Delta_{I_{\vec{b}},m'}$. Thus $m'=m$ and $I = I_{\vec{b}}$.
\end{proof}
 
The uniqueness of the maximal destabilizing subbundle can be shown as in the case of ordinary bundles. 

Suppose that $\vec{a}$ is a general point of $\Delta_{I, m}$. Let $\vec{a}^{-}$ (resp. $\vec{a}^{+}$) be in $\Delta_{I, m}^{-}$ (resp. $\Delta_{I, m}^{+}$). Assume that both $\cM(\vec{a}^{-})$ and $\cM(\vec{a}^{+})$ are nonempty. Since $\vec{a}^{\pm}$ are general parabolic weights, $\cM(\vec{a}^{\pm})$ are smooth by Corollary \ref{smoothness of moduli}. There are two birational morphisms
\[
	\xymatrix{\cM(\vec{a}^{-}) \ar[rd]^{\phi^{-}} &&
	\cM(\vec{a}^{+}) \ar[ld]_{\phi^{+}}\\
	& \cM(\vec{a}).}
\]
The image $Y$ of the exceptional locus of $\phi^{\pm}$ is the locus parameterizes S-equivalent classes of $(\cO(-m), \vec{b}) \oplus (\cO(m), \vec{c})$, where $(\cO(-m), \vec{b})$ is the destabilizing bundle for $\Delta_{I, m}^{+}$ and $\vec{c} = \vec{a} - \vec{b}$. The moduli of parabolic line bundles of a fixed degree on $\PP^{1}$ is a point because there is a unique line bundle for each degree. So $Y$ is always a single point. For the same $I$, define $\vec{b}^{\pm}$ and $\vec{c}^{\pm}$ by using $\vec{a}^{\pm}$. The exceptional fiber of $\phi^{-}$ (resp. $\phi^{+}$) is a projective space $Y^{-} := \PP \Ext^{1}((\cO(m), \vec{c}^{-}), (\cO(-m), \vec{b}^{-}))$ (resp. $Y^{+} := \PP \Ext^{1}((\cO(-m), \vec{b}^{+}), (\cO(m), \vec{c}^{+}))$). 

Using Proposition \ref{ext of par line bdls} and arguments in \cite[Section 5]{Tha02}, we can see that $Y^{-}\cong\cM(\vec{a}^{-})\setminus\cM(\vec{a}^{+})$ and $Y^{+}\cong\cM(\vec{a}^{+})\setminus\cM(\vec{a}^{-})$.

\begin{proposition}[\protect{\cite[Section 7]{Tha96}}]\label{prop:wallcrossing}
The blow-up of $\cM(\vec{a}^{-})$ along $Y^{-}$ is isomorphic to the blow-up of $\cM(\vec{a}^{+})$ along $Y^{+}$. In particular, $\dim Y^{-} + \dim Y^{+} = \dim \cM(\vec{a}) - 1$. 
\end{proposition}

We will use the following dimension computation later. 

\begin{proposition}\label{dim of centers}
Let $\vec{a}^{-}$ be a general point of $\Delta_{I, m}^{-}$. Then 
\[
	\dim \Ext^1((\cO(m),\vec{c}^{-}),(\cO(-m),\vec{b}^{-})) 
	= 2m+n-1-|I|.
\]
\end{proposition}

\begin{proof}
By Proposition \ref{Serre duality}, we have natural isomorphisms
\[
	\Ext^1((\cO(m),\vec{c}^{-}),(\cO(-m),\vec{b}^{-}))\cong
	\spHom((\cO(-m-(n-2)),\vec{b}^{-}),(\cO(m),\vec{c}^{-}))^{\vee}.
\]
Consider the following short exact sequence of sheaves
\[
	0\to\cspHom((\cO(-m-(n-2)),\vec{b}^{-}),(\cO(m),\vec{c}^{-}))\to
	\cHom((\cO(-m-(n-2)),\vec{b}^{-}),(\cO(m),\vec{c}^{-}))
\]
\[
	\to\frac{\bigoplus_{i=1}^{n}\Hom((\cO(-m-(n-2)),\vec{b}^{-})|_{p_i},
	(\cO(m),\vec{c}^{-})|_{p_i})}
	{\bigoplus_{i=1}^{n}N_{p_i}((\cO(-m-(n-2)),\vec{b}^{-}),
	(\cO(m),\vec{c}^{-}))}\to 0
\]
where $N_p((\cO(k),\vec{x}),(\cO(\ell),\vec{y}))$ is the subspace of strictly parabolic maps in $\Hom((\cO(k),\vec{x})|_{p},(\cO(\ell),\vec{y})|_{p})$ at a point $p\in\PP^1$. For $\vec{a}^{-} \in \Delta_{I, m}^{-}$, $\mu(\cO(m), \vec{c}^{-}) > \mu(\cO(-m), \vec{b}^{-})$. Thus
\[
	\mathrm{H}^1(\cspHom((\cO(-m-(n-2)),\vec{b}^{-}),(\cO(m),\vec{c}^{-})))
	= \Ext^{0}((\cO(m),\vec{c}^{-}),(\cO(-m),\vec{b}^{-}))^{\vee}
\]
\[
	=\pHom((\cO(m),\vec{c}^{-}),(\cO(-m),\vec{b}^{-}))^{\vee}
	=0
\]
by Proposition \ref{mor of stable par bdls}. Hence we have a short exact sequence of vector spaces
\[
	0\to\spHom((\cO(-m-(n-2)),\vec{b}^{-}),(\cO(m),\vec{c}^{-}))\to
	\Hom((\cO(-m-(n-2)),\vec{b}^{-}),(\cO(m),\vec{c}^{-}))
\]
\[
	\to\frac{\oplus_{i=1}^{n}\Hom((\cO(-m-(n-2)),\vec{b}^{-})|_{p_i},
	(\cO(m),\vec{c}^{-})|_{p_i})}{\oplus_{i=1}^{n}N_{p_i}
	((\cO(-m-(n-2)),\vec{b}^{-}),(\cO(m),\vec{c}^{-}))}\to 0.
\]
Since
\[
	\dim N_{p_i}((\cO(-m-(n-2)),\vec{b}^{-}),(\cO(m),\vec{c}^{-}))
	=\left\{\begin{matrix}0,&i\in I\\1,&i\in I^c\end{matrix}\right.,
\]
\[
	\dim\spHom((\cO(-m-(n-2)),\vec{b}^{-}),(\cO(m),\vec{c}^{-}))
	=\dim\Hom((\cO(-m-(n-2)),\vec{b}^{-}),(\cO(m),\vec{c}^{-}))-|I|
\]
\[
	=\dim \mathrm{H}^0(\cO(2m+n-2))-|I|=2m+n-1-|I|.
\]
\end{proof}

\section{Elementary GIT quotients and the moduli space of parabolic bundles}\label{sec:GIT}

The ring of invariants of a product of projective lines have been studied since the 19th century. In this section, we review some of the classical results and its relation with moduli spaces of rank 2 parabolic vector bundles on $\PP^{1}$. For the basic of GIT, consult \cite{GIT94}.

\subsection{The GIT quotient of a product of projective lines}

Fix $n \ge 3$. For $\vec{a} = (a_{1}, \cdots, a_{n}) \in \QQ_{> 0}^{n}$, consider an ample $\QQ$-line bundle $L := \cO(a_{1}, \cdots, a_{n})$ on $(\PP^{1})^{n}$. On $(\PP^{1})^{n}$, $\SL_{2}$ acts diagonally. We can take the GIT quotient with respect to $L$,
\[
	(\PP^{1})^{n}\git_{L}\SL_{2} := \proj \bigoplus_{m \ge 0}
	H^{0}((\PP^{1})^{n}, \lfloor L^{m}\rfloor)^{\SL_{2}}.
\]

Conditions for the (semi)stability of $(\PP^{1})^{n}$ with respect to $L$ are described in the following theorem. We denote the stable (resp. semistable) locus by $((\PP^{1})^{n})^{s}$ (resp. $((\PP^{1})^{n})^{ss}$). 

\begin{theorem}[\protect{\cite[Proposition 3.4]{GIT94}}]\label{thm:stability}
Let $L = \cO(a_{1}, \cdots, a_{n})$ be a $\QQ$-linearization. Let $a := \sum_{i=1}^{n}a_{i}$. For a point $x := (x_{1}, \cdots, x_{n}) \in (\PP^{1})^{n}$, $x \in ((\PP^{1})^{n})^{ss}$ (resp. $x \in ((\PP^{1})^{n})^{s}$) if and only if for any $y \in \PP^{1}$, 
\[
	\sum_{x_{i}=y}a_{i} \le a/2 \;(\mbox{resp.} < a/2).
\]
\end{theorem}

\begin{corollary}\label{cor:stability}
\begin{enumerate}
\item For a linearization $L = \cO(a_{1}, \cdots, a_{n})$, $(\PP^{1})^{n}\git_{L}\SL_{2}$ is nonempty if and only if $a_{i} \le a/2$ for every $1 \le i \le n$.
\item The stable locus is nonempty (in particular, $(\PP^{1})^{n}\git_{L}\SL_{2}$ is $(n-3)$-dimensional) if and only if $a_{i} < a/2$ for every $1 \le i \le n$.
\item The semi-stable locus coincides with the stable locus if and only if for any nonempty $I \subset [n]$, $\sum_{i \in I}a_{i} \ne \sum_{i \notin I}a_{i}$.
\end{enumerate}
\end{corollary}

\begin{definition}\label{def:effectivelinearizatiqon}
We say that a linearization $L$ is \textbf{effective} if it satisfies (2). An effective linearization is \textbf{general} if it satisfies (3) as well. Compare with Definition \ref{def:effecivegeneralweights}. 
\end{definition}

\begin{remark}\label{rmk:smoothnessofgeneralquotient}
The subgroup $\{\pm 1\} \subset \SL_{2}$ acts trivially on $(\PP^{1})^{n}$, thus the $\SL_{2}$-action induces a $\mathrm{PGL}_{2} = \SL_{2}/\{\pm 1\}$-action. If $L$ is general, at each point $x \in ((\PP^{1})^{n})^{s}$ the stabilizer $\mathrm{Stab}_{x}$ is $\{\pm 1\} \subset \SL_{2}$. So $\mathrm{PGL}_{2}$ acts on $((\PP^{1})^{n})^{s}$ freely, and $(\PP^{1})^{n}\git_{L}\SL_{2}$ is smooth. 
\end{remark}

\subsection{The moduli space of parabolic bundles as an elementary GIT quotient}

The readers are able to observe that the combinatorics of the GIT stability is identical to that of the stability of rank 2 parabolic bundles on $\PP^{1}$.

\begin{proposition}\label{prop:parabolicGIT}
Let $\vec{a} = (a_{1}, \cdots, a_{n}) \in W$ and let $L = \cO(a_{1}, \cdots, a_{n})$ be the corresponding $\QQ$-linearization. Assume that $L$ is effective and $a := \sum_{i = 1}^{n}a_{i} < 2$. Then 
\[
    \cM(\vec{a}) \cong (\PP^{1})^{n}\git_{L}\SL_{2}.
\]
\end{proposition}

\begin{proof}
First of all, let $(E, \{V_{i}\})$ be a semistable parabolic bundle of degree 0. By Grothendieck's theorem (\cite[Theorem 1.3.1]{HuLe10}), $E = \cO(k) \oplus \cO(-k)$ for some nonnegative integer $k$. If $k \ge 1$, then $\mu(\cO(k)) \ge 1 > a/2 = \mu(E)$. Thus $E$ is not semistable unless it is a trivial bundle.

Let $X = (\PP^{1})^{n}$ and $\pi_{i} : X \to \PP^{1}$ be the $i$-th projection. Let $\cE$ be a rank two trivial vector bundle on $X \times \PP^{1}$. Then $\PP(\cE)$ is isomorphic to $X \times \PP^{1} \times \PP^{1}$.  For each $i$, define a morphism
\[
	s_{i}: X \to X \times \PP^{1} \stackrel{\cong}{\to} X\times\{p_i\}\times\PP^1
\]
by the graph of $i$-th projection. Over each $X \times \{p_{i}\}$, define a line bundle $\cV_{i}$ as $[\cV_{i}|_{(x, p_{i})}] = s_{i}(x)$. Then $\cV_{i}$ is a natural subbundle of $\cE|_{X \times \{p_{i}\}}$ of rank one. Now $(\cE, \{\cV_{i}\})$ is a family of rank 2 parabolic vector bundles on $\PP^{1}$ over $X$. Consider the restricted family over $X^{ss}$ and use the same notation $(\cE, \{\cV_{i}\})$.

Let $(\cO^2, \{V_{i}\})$ be the fiber of $(\cE, \{\cV_{i}\})$ over $x = (x_{1}, \cdots, x_{n}) \in X^{ss}$. Note that all subbundles of $\cO^{2}$ is $\cO(-k)$ for some nonnegative integer $k$. If $k \ge 1$, then
\[
	\mu(\cO(-k)) = -k +  \sum_{\cO(-k)|_{p_{i}} = V_{i}}a_{i} <
	\frac{a}{2} = \mu(\cO^{2})
\]
because 
\[
	\sum_{\cO(-k)|_{p_{i}} = V_{i}}a_{i}- \sum_{\cO(-k)|_{p_{i}} \neq V_{i}}a_{i} 
	\le a < 2 \le 2k.
\]
So it is not a destabilizing bundle. 

Let $F = V \otimes \cO \subset \cO^{2}$ be a trivial subbundle for a one dimensional subspace $V \subset \CC^{2}$. Note that if $F|_{x_{i}} = V$ for some $i \in I \subset [n]$, $x_{i} = x_{j}$ for every $i, j \in I$. From the GIT stability in Theorem \ref{thm:stability} (with $y = [V]$), $\sum_{E|_{x_{i}} = V} a_{i} \le a/2$. Thus for $x = (x_{1}, \cdots, x_{n})\in X^{ss}$,
\[
	\mu(E) = \sum_{E|_{x_{i}} = V}a_{i} \le \frac{a}{2} = \mu(\cO^{2}).
\]
Therefore $X^{ss}$ parametrizes semistable parabolic vector bundles with respect to the parabolic weight $\vec{a}$. Thus we have a classifying morphism $\mu : X^{ss} \to \cM(\vec{a})$. There is a natural $\SL_{2}$-action on $X^{ss}$ and each orbit parametrizes isomorphic parabolic bundles, since it acts as a canonical $\SL_{2}$-action on each fiber of the trivial rank 2 bundle. Thus there is a quotient morphism $\bar{\mu} : X\git_{L}\SL_{2} \to \cM(\vec{a})$.

One can check that $\bar{\mu}$ is injective. Indeed, the injectivity over the stable locus $\cM(\vec{a})^{s}$ is obvious. For a strictly semistable point corresponding an S-equivalent class of $E := (\cO, \vec{b}_{1})\oplus (\cO, \vec{b}_{2}) \in \cM(\vec{a})$, $\mu^{-1}(E) = X_{1} \cup X_{2}$ where $X_{i} = \{(x_{1}, \cdots, x_{n})\;|\; x_{j} = x_{k} \mbox{ if } j, k \in I_{\vec{b}_{i}}\}$. Because the closure of the orbit of a point in $X_{i}$ contains an orbit $X_{1} \cap X_{2}$ which is closed in $X^{ss}$, they are identified to a point in the GIT quotient. Since $\cM(\vec{a})$ is irreducible and $\bar{\mu}$ is dominant, $\bar{\mu}$ is surjective. Finally, because $\cM(\vec{a})$ is normal by Theorem \ref{moduli par bdls}, $\bar{\mu}$ is an isomorphism.
\end{proof}

It is already known that $\cM(\vec{a})$ is rational for any effective parabolic weight $\vec{a}\in W$ (\cite{Bau91, BH95}). We provide another proof of the rationality of $\cM(\vec{a})$ for any effective parabolic weight $\vec{a}\in W$, which is a simple consequence of Proposition \ref{prop:parabolicGIT}.

\begin{corollary}
For any effective parabolic weight $\vec{a} \in W$, $\cM(\vec{a})$ is rational.
\end{corollary}
\begin{proof}
By Proposition \ref{prop:wallcrossing} and Proposition \ref{prop:parabolicGIT}, $\cM(\vec{a})$ is birational to $(\PP^{1})^{n}\git_{L}\SL_{2}$ where $L$ is an effective linearization. It is known that $(\PP^{1})^{n}\git_{L}\SL_{2}\cong\PP^{n-3}$, when $L$ is proportional to $\cO(n-2,1,1,\cdots,1)$ (\cite[Sections 6.2, 7.2 and Theorem 8.2]{Has03}).
\end{proof}

\subsection{General case}

For a general parabolic weight $\vec{a} \in W$, we may find $c > 1$ such that $\vec{a} = c\vec{b}$ and $\sum b_{i} < 2$. Thus to study the geometry of $\cM(\vec{a})$, it suffices to study the change of the moduli space when the parabolic weight changes from $\cM(\vec{b})$ to $\cM(c\vec{b}) = \cM(\vec{a})$ for $c > 1$. Note that if $1 \le c < \min\{1/b_{i}\}$, then $c\vec{b} \in W$, too. By perturbing the given parabolic weight slightly, we may assume that all wall-crossings are simple ones. 

\begin{proposition}\label{prop:generalmoduli}
Let $\vec{a}$ be a general parabolic weight in $W$ such that $\sum a_{i} < 2$. Consider the wall-crossings from $\cM(\vec{a})$ to $\cM(c\vec{a})$ as $c$ increases in the range of $1 \le c < \min\{1/a_{i}\}$. Suppose that all wall-crossings are simple ones. Then the first wall-crossing is a blow-up at the point $[\vec{p}] \in (\PP^{1})^{n}\git_{L}\SL_{2} \cong \cM(\vec{a})$. All other wall-crossings are flips or blow-downs. 
\end{proposition}

\begin{proof}
By Lemma \ref{lem:stabilitywall}, each stability wall is given by $\Delta_{I, m}$. Then for $c_{0}\vec{a} \in \Delta_{I, m}$, 
\[
	c_{0}\left(\sum_{i \in I}a_{i} - \sum_{i \in I^{c}}a_{i}\right) = 2m
\]
by \eqref{eqn:stabilitywall}. Then for $c > c_{0}$, $c\vec{a} \in \Delta_{I, m}^{+}$ and 
\[
	c\left(\sum_{i \in I}a_{i} - \sum_{i \in I^{c}}a_{i}\right) > 2m.
\]
This is true only if $m \ge 0$. Also, during the variation of stability conditions in the proposition, we do not meet a stability wall of type $\Delta_{I, 0}$, because the ratios between parabolic weights do not change. Thus $m > 0$. Now it is clear that the first stability wall that we meet is $\Delta_{[n], 1}$, i.e., when $\sum_{i=1}^{n} ca_{i} = 2$. 

Then by Proposition \ref{prop:wallcrossing}, the blow-up of $\cM^{-}$ along $Y^{-}$ is isomorphic to the blow-up of $\cM^{+}$ along $Y^{+}$. Furthermore, by Proposition \ref{dim of centers}, $Y^{-}$ is a point and $Y^{+}=\PP^{n-4}$, which is a divisor of $\cM^{+}$. Therefore $\cM^{+}$ is isomorphic to the blow-up of $\cM^{-} = (\PP^{1})^{n}\git_{L}\SL_{2}$ at the point $Y^{-}$. Note that $x = (x_{1}, \cdots, x_{n}) \in Y^{-}$ if and only if the corresponding parabolic bundle $(\cO^{2}, \{V_{i}\})$ has a subbundle $\cO(-1)$ which contains all $V_{i}$'s. Since $\cO(-1) \subset \cO$ is isomorphic to the tautological subbundle, $(x_{1}, \cdots, x_{n}) = ([V_{1}], \cdots, [V_{n}])$ is equivalent to $\vec{p} = (p_{1}, \cdots, p_{n})$. 

After the first wall-crossing, since $m > 1$ or $|I| < n$, $2m+n-1-|I| > 1$. Thus by Proposition \ref{dim of centers} and Proposition \ref{prop:wallcrossing} again, the modification is not a blow-up anymore. 
\end{proof}

\section{The effective cone of the GIT quotient}\label{sec:effectiveconeGIT}

As a first step toward Mori's program of $\cM(\vec{a})$ and $(\PP^{1})^{n}\git_{L}\SL_{2}$, we compute the effective cone of $(\PP^{1})^{n}\git_{L}\SL_{2}$. 

\subsection{Rational Picard group}

The Picard group of $(\PP^{1})^{n}$ is generated by the pull-backs $\pi_{i}^{*}\cO(1)$ for $1 \le i \le n$ where $\pi_{i} : (\PP^{1})^{n} \to \PP^{1}$ is the $i$-th projection. We denote the tensor product $\pi_{1}^{*}\cO(b_{1}) \otimes \cdots \otimes \pi_{n}^{*}\cO(b_{n})$ by $\cO(b_{1}, \cdots, b_{n})$, or $\cO(\sum_{i=1}^{n}b_{i}e_{i})$ where $e_{i}$ is the $i$-th standard basis in $\QQ^{n}$. So $\Pic((\PP^{1})^{n}) \cong \ZZ^{n}$ and the nef cone $\mathrm{Nef}((\PP^{1})^{n}) \subset \Pic((\PP^{1})^{n})_{\QQ} \cong \QQ^{n}$ is generated by $\cO(e_{i})$. The effective cone is equal to the nef cone so it is simplicial. 

Let $L = \cO(a_{1}, \cdots, a_{n})$ be a $\QQ$-linearization of $(\PP^{1})^{n}$. Consider the GIT quotient $(\PP^{1})^{n}\git_{L}\SL_{2}$. Since it is a quotient of semistable locus, there is a natural diagram
\[
	\xymatrix{((\PP^{1})^{n})^{ss} \ar[r]^-{\iota} \ar[d]_{\pi}
	& (\PP^{1})^{n}\\ (\PP^{1})^{n}\git_{L}\SL_{2}}
\]
where $\iota$ is the inclusion and $\pi$ is the quotient map.

For any two indices $1 \le i < j \le n$, let 
\[
	\Delta_{\{i,j\}}= \{(x_{1}, \cdots, x_{n}) \in (\PP^{1})^{n}
	\;|\; x_{i} = x_{j}\}.
\]
It is $\SL_{2}$-invariant, so it descends to an effective cycle
\[
	D_{\{i,j\}} = \pi(\iota^{*}(\Delta_{\{i,j\}})) = 
	\{(x_{1}, \cdots, x_{n})\in (\PP^{1})^{n}\git_{L}\SL_{2}
	\;|\; x_{i} = x_{j}\}
\]
on the quotient, if $\Delta_{\{i,j\}}$ intersects the semistable locus, i.e., $a_{i}+a_{j} \le a/2$. Furthermore, if it intersects the stable locus (so $a_{i}+a_{j} < a/2$), then $D_{\{i,j\}}$ is a divisor on $(\PP^{1})^{n}\git_{L}\SL_{2}$. If $a_{i}+a_{j} = a/2$, $\Delta_{\{i, j\}}$ has a unique semistable orbit $\{(x_{1}, \cdots, x_{n}) \;|\; x_{i}=x_{j}, x_{k} = x_{\ell} \mbox{ for all } k, \ell \ne i, j\}$ which is closed in $((\PP^{1})^{n})^{ss}$. Thus in this case $D_{\{i,j\}}$ is a single point. 

Note that $\cO(\Delta_{\{i,j\}}) = \cO(e_{i}+e_{j})$. 

\begin{proposition}\label{prop:picardgroup}
Suppose that $n \ge 5$. Let $L = \cO(a_{1}, \cdots, a_{n})$ be a general $\QQ$-linearization on $(\PP^{1})^{n}$. Let $a = \sum a_{i}$.
\begin{enumerate}
\item The rational Picard group $\Pic((\PP^{1})^{n}\git_{L}\SL_{2})_{\QQ}$ is naturally identified with the quotient space 
\[
	\Pic((\PP^{1})^{n})_{\QQ}/\langle \Delta_{\{i,j\}}\;|\; 
	a_{i}+a_{j} \ge a/2\rangle,
\]
via the identification $D_{\{i,j\}} \mapsto \Delta_{\{i,j\}}$.
\item The rank of $\Pic((\PP^{1})^{n}\git_{L}\SL_{2})_{\QQ}$ is $n - k$, where $k$ is the number of $\Delta_{\{i,j\}}$ with $a_{i}+a_{j} \ge a/2$.
\end{enumerate}
\end{proposition}

\begin{remark}
When $n = 4$, for any effective linearization $L$, the GIT quotient $(\PP^{1})^{4}\git_{L}\SL_{2}$ is isomorphic to $\PP^{1}$.
\end{remark}

\begin{proof}[Proof of Proposition \ref{prop:picardgroup}]
Let $((\PP^{1})^{n})^{us} := (\PP^{1})^{n} - ((\PP^{1})^{n})^{s}$ be the unstable locus and let $j : ((\PP^{1})^{n})^{us} \hookrightarrow (\PP^{1})^{n}$ be the inclusion. We have a natural exact sequence
\[
	A_{n-1}(((\PP^{1})^{n})^{us}) \stackrel{j_{*}}{\longrightarrow}
	\Pic((\PP^{1})^{n}) \stackrel{\iota^{*}}{\longrightarrow}
	\Pic(((\PP^{1})^{n})^{s}) \to 0.
\]
After tensoring $\QQ$, the sequence is exact too. Each $(n-1)$-dimensional irreducible component of $((\PP^{1})^{n})^{us}$ is of the form $\Delta_{\{i,j\}}$ with $a_{i}+a_{j} \ge a/2$. Therefore we have 
\[
	\Pic(((\PP^{1})^{n})^{s})_{\QQ}
	= \Pic((\PP^{1})^{n})_{\QQ}/\langle \Delta_{\{i,j\}}\;|\; 
	a_{i}+a_{j} \ge a/2\rangle.
\]

Let $\Pic(X)^{\SL_{2}}$ be the group of isomorphism classes of $\SL_{2}$-invariant line bundles on $X$. Since
\[
	\cO(e_{i}) = \frac{1}{2}\left(\cO(e_{i}+e_{j}) \otimes  
	\cO(e_{i}+e_{k}) \otimes \cO(e_{j}+e_{k})^{-1}\right) = 
	\frac{1}{2}\left(\cO(\Delta_{\{i,j\}}+\Delta_{\{i,k\}}
	- \Delta_{\{j,k\}})\right)
\]
and the right hand side is $\SL_{2}$-invariant, $\Pic((\PP^{1})^{n})_{\QQ} \cong \Pic((\PP^{1})^{n})_{\QQ}^{\SL_{2}}$. The same identity is true for $((\PP^{1})^{n})^{s}$, too.

By Kempf's descent lemma (\cite[Theorem 2.3]{DN89}), an $\SL_{2}$-linearized (in particular, $\SL_{2}$-invariant) line bundle $E$ on $((\PP^{1})^{n})^{s}$ descends to $(\PP^{1})^{n}\git_{L}\SL_{2}$ if and only if for every closed orbit $\SL_{2}\cdot x$, the stabilizer $\mathrm{Stab}_{x}$ acts on $E_{x}$ trivially. Because $\mathrm{Hom}(\SL_{2}, \CC^{*})$ is trivial, for each $\SL_{2}$-invariant line bundle there is at most one linearization. Furthermore since $(\PP^{1})^{n}$ is normal, for any $\SL_{2}$-invariant line bundle $E$, $E^{n}$ admits a linearization for some $n \in \NN$ (\cite[Corollary I.1.6]{GIT94}). Therefore 
\[
	\Pic((\PP^{1})^{n}\git_{L}\SL_{2})_{\QQ}
	\cong \Pic(((\PP^{1})^{n})^{s})_{\QQ}^{\SL_{2}}.
\]
This isomorphism is given by $\pi^{*}$. Thus $D_{\{i,j\}}$ maps to $\Delta_{\{i,j\}}$. This proves Item (1).

To show Item (2), it suffices to show that the set of divisorial unstable components are linearly independent in $\Pic((\PP^{1})^{n})_{\QQ}$. Let $G$ be a finite simple graph with vertex set $[n]$ and edge set $\{\Delta_{\{i,j\}}\;|\; a_{i}+a_{j} \ge a/2\}$, the set of unstable divisors. Two vertices $i$ and $j$ are connected by $\Delta_{\{i,j\}}$. If there are two disjoint edges $\Delta_{\{i,j\}}, \Delta_{\{k, \ell\}}$ in $G$, $a > a_{i}+a_{j} + a_{k} + a_{\ell} \ge a$. Thus there are no disjoint edges. Then $G$ must be a star shaped graph (all vertices are connected to a central vertex) or a complete graph $K_{3}$ of degree 3. In these cases, it is straightforward to check that the edge set is linearly independent. 
\end{proof}

\begin{definition}
A $\QQ$-linearization $L = \cO(a_{1}, \cdots, a_{n})$ is called a linearization with \textbf{a maximal stable locus} if $a_{i}+a_{j} < a/2$ for any $\{i,j\} \subset [n]$. 
\end{definition}

Note that if $L$ is a linearization with a maximal stable locus, then every irreducible component of the unstable locus has codimension at least two. In particular, we have the maximal possible Picard rank. It includes the case of \textbf{symmetric} linearization $L = \cO(b, b, \cdots, b)$ for some $b \in \QQ_{> 0}$. 

\begin{corollary}
Suppose that $n \ge 5$. Let $L$ be a general $\QQ$-linearization with a maximal stable locus. Then $\Pic((\PP^{1})^{n}\git_{L}\SL_{2})_{\QQ}$ is isomorphic to $\Pic((\PP^{1})^{n})_{\QQ}$. In particular, it has rank $n$ and $D_{\{i, j\}}$ generates the rational Picard group.
\end{corollary}

\subsection{The effective cone}

The following proposition is a translation of a result in the classical invariant theory.

\begin{proposition}\label{prop:effectiveconeGIT}
Suppose that $n \ge 5$. Let $L$ be a general $\QQ$-linearization $\cO(a_{1}, \cdots, a_{n})$ on $(\PP^{1})^{n}$. Then the effective cone $\mathrm{Eff}((\PP^{1})^{n}\git_{L}\SL_{2})$ of $(\PP^{1})^{n}\git_{L}\SL_{2}$ is generated by $\{D_{\{i,j\}}\;|\; 1 \le i < j \le n, a_{i}+a_{j} < a/2\}$. 
\end{proposition}

\begin{proof}
Let $D$ be an effective divisor on $(\PP^{1})^{n}\git_{L}\SL_{2}$. Then $\pi^{*}D$ is an $\SL_{2}$-invariant divisor on $((\PP^{1})^{n})^{s}$. By taking its closure in $(\PP^{1})^{n}$, we have an $\SL_{2}$-invariant divisor $\Delta := \overline{\pi^{*}(D)}$ on $(\PP^{1})^{n}$. Then none of irreducible components of $\Delta$ is in $\{\Delta_{\{i,j\}}\;|\; 1\le i < j \le n, a_{i}+a_{j} \ge a/2\}$, since they are disjoint from $((\PP^{1})^{n})^{s}$.

If we denote the homogeneous coordinates of the $i$-th factor of $(\PP^{1})^{n}$ by $[s_{i}:t_{i}]$, then by the first fundamental theorem of invariant theory (\cite[Section 2]{HMSV09a}), for any line bundle $E$ on $(\PP^{1})^{n}$, every $\SL_{2}$-invariant element of $\mathrm{H}^{0}(E)$ is generated by products of $(s_{i}t_{j}-s_{j}t_{i})$, which is precisely $\Delta_{\{i,j\}}$. In particular, $\Delta \in \mathrm{H}^{0}(\cO(\Delta))^{\SL_{2}}$ is an effective linear combination $\sum c_{\{i,j\}}\Delta_{\{i,j\}}$ for some $\Delta_{\{i,j\}}$ and $c_{\{i,j\}} > 0$. Moreover, on this linear combination $\Delta_{\{i,j\}}$ with $a_{i}+a_{j} \ge a/2$ does not appear since they are unstable and $\Delta$ does not have such components. Now each $\Delta_{\{i,j\}}$ descends to $D_{\{i,j\}}$. Thus $D = \sum c_{\{i,j\}}D_{\{i,j\}}$ with $a_{i}+a_{j} < a/2$. 

In summary, every effective divisor on $(\PP^{1})^{n}\git_{L}\SL_{2}$ is an effective linear combination of $\{D_{\{i,j\}}\;|\; 1 \le i < j \le n, a_{i}+a_{j} < a/2\}$.
\end{proof}

\begin{corollary}
Let $L = \cO(a_{1}, \cdots, a_{n})$ be a general $\QQ$-linearization with a maximal stable locus. Then $\mathrm{Eff}((\PP^{1})^{n}\git_{L}\SL_{2})$ has precisely $2n$ facets, namely,
\[
	P_{i} := \mathrm{Span}\{D_{\{i,j\}}\;|\; j \ne i\}, \quad 1 \le i \le n
\]
and 
\[
	N_{i} := \mathrm{Span}\{D_{\{j,k\}}\;|\; j, k \ne i\}, \quad 1 \le i \le n.
\]
\end{corollary}

\begin{proof}
Take the hyperplane section $\sum a_{i} = 2$ in $\Pic((\PP^{1})^{n}\git_{L}\SL_{2})_{\QQ} \cong \Pic((\PP^{1})^{n})_{\QQ}$. Then the intersection with the effective cone generated by $\{D_{\{i,j\}}\}$ is the hypersimplex 
\[
	\Delta(2, n) = \{(a_{1}, \cdots, a_{n}) \in \QQ^{n}\;|\; 
	\sum_{i=1}^{n} a_{i} = 2, 0 \le a_{i} \le 1\}
\]
(\cite[Section 1]{Kap93b}). There is a one-to-one correspondence between the set of facets of $\mathrm{Eff}((\PP^{1})^{n}\git_{L}\SL_{2})$ and that of $\Delta(2, n)$. Now the statement follows from \cite[Proposition 1.2.5]{Kap93b}.
\end{proof}

\begin{remark}\label{rem:dualcurve}
When $L$ is a general linearization with a maximal stable locus, the construction of the dual curve for each facet of $\mathrm{Eff}((\PP^{1})^{n}\git_{L}\SL_{2})$ is easy. Since $(\PP^{1})^{n}\git_{L}\SL_{2}$ is naturally a moduli space of $n$-pointed smooth rational curves (\cite[Section 8]{Has03}), it suffices to construct a one-dimensional family of $n$-pointed smooth rational curves with appropriate stability condition described by $L$. 

First of all, take $n-1$ general lines $\ell_{2}, \cdots, \ell_{n}$ on $\PP^{2}$. Take a general point $x \in \PP^{2}- \cup \ell_{i}$. Blow-up $\PP^{2}$ at $x$  and let $\ell_{1}$ be the exceptional divisor. Then $\mathrm{Bl}_{x}\PP^{2} \cong \mathbb{F}_{1}$ is a $\PP^{1}$-bundle over $\ell_{1}$ and we can regard it as a family of $n$-pointed smooth rational curves on $C_{1} := \ell_{1}$. Because $a_{i}+a_{j} < a/2$, any two marked points can collide. Thus all fibers are stable. So $C_{1}$ is a curve on $(\PP^{1})^{n}\git_{L}\SL_{2}$. Then $C_{1} \cdot D_{\{1,j\}} = 0$ and $C_{1} \cdot D_{\{i,j\}} = 1$ for $2 \le i, j \le n$. Therefore $C_{1}$ is a dual curve for $P_{1}$.

Now consider a trivial family $\pi : \PP^{1} \times \PP^{1} \to \PP^{1}$ with $(n-1)$ distinct constant sections $\sigma_{2}, \cdots, \sigma_{n}$, and a diagonal section $\sigma_{1}$. Then $(\pi : \PP^{1} \times \PP^{1} \to \PP^{1}, \sigma_{1}, \cdots, \sigma_{n})$ is a family of pointed curves over $B_{1} := \PP^{1}$. So $B_{1}$ is a curve on $(\PP^{1})^{n}\git_{L}\SL_{2}$. Now $B_{1} \cdot D_{\{1,j\}} = 1$ and $B_{1} \cdot D_{\{i,j\}} = 0$ for $2 \le i, j \le n$. Therefore $B_{1}$ is the dual curve for $N_{1}$. 
\end{remark}

We close this section with a new notation for line bundles on $(\PP^{1})^{n}\git_{L}\SL_{2}$.

\begin{definition}
Suppose that $n \ge 5$. Let $L$ be a general linearization. We denote a $\QQ$-line bundle $E$ on $(\PP^{1})^{n}\git_{L}\SL_{2}$ by $\overline{\cO}(b_{1}, \cdots, b_{n})$ (or $\overline{\cO}(\sum b_{i}e_{i})$) if $E$ maps to the equivalent class of $\cO(b_{1}, \cdots, b_{n}) = \cO(\sum b_{i}e_{i})$ under the isomorphism
\[
	\Pic((\PP^{1})^{n}\git_{L}\SL_{2})_{\QQ} \cong 
	\Pic((\PP^{1})^{n})_{\QQ}/\langle \Delta_{\{i,j\}}\;|\;
	a_{i} + a_{j} \ge a/2\rangle
\]
in Proposition \ref{prop:picardgroup}. 
\end{definition}

Note that if there is an unstable divisor, the expression is not unique. For instance, if $\Delta_{\{1,2\}}$ is unstable, $\overline{\cO}(b_{1}, \cdots, b_{n}) = \overline{\cO}(b_{1}+c, b_{2}+c, b_{3}, \cdots, b_{n})$ for any $c \in \QQ$.

\section{Background on conformal blocks}\label{sec:conformalblock}

In last three decades, the space of conformal blocks, which are fundamental objects in conformal field theory, have been studied by many mathematicians and physicists. Although the original construction is using the representation theory of affine Lie algebras, in this section we give an elementary definition of the simplest case - $\mathfrak{sl}_{2}$ conformal blocks on $\PP^{1}$ - and their algebraic/combinatorial realizations. Because we do not give the usual definition, we leave some references for the reader's convenience. For the general definition of conformal blocks, see \cite{Uen08}. The connection with the moduli space of parabolic vector bundles, see \cite{Pau96}. 

\subsection{A quick definition of $\mathfrak{sl}_{2}$ conformal blocks}
\label{ssec:quickdef}

In this section, we review an elementary definition of $\mathfrak{sl}_{2}$ conformal blocks on $\PP^{1}$, described in \cite[Section 1]{Loo09}. For the equivalence of the following definition and the original one, consult \cite[Proposition 4.1]{Bea96a}. 

We begin with some notational conventions. In this section, we write a sequence $(k_{1}, \cdots, k_{n})$ as $\mathbf{k}$. $|\mathbf{k}| = \sum_{i}k_{i}$ and $\mathbf{k}! = \prod_{i}k_{i}!$. 

For any nonnegative integer $k$, let $V_{k} = \mathrm{H}^{0}(\PP^{1}, \cO(k))$ be an irreducible $\SL_{2}$-representation with highest weight $k$. The vector space $V_{k}$ is identified with $\CC[x,y]_{k}$, the space of homogeneous polynomials of degree $k$. The infinitesimal $\mathfrak{sl}_{2}$-action on $\CC[x,y]_{k}$ is given by $e = x\partial_{y}, f = y\partial_{x}, h = x\partial_{x} - y \partial_{y}$ for the standard basis $e, f, h$ of $\mathfrak{sl}_{2}$. The highest weight vector of $V_{k}$ is $x^{k}$ and $f^{j}x^{k} = \frac{k!}{(k-j)!}x^{k-j}y^{j}$. We may dehomogenize it by taking $x = 1$. Then $V_{k}$ is identified $\CC[y]_{\le k}$ (the space of polynomials of degree at most $k$) and the action of $e$ is given by $\partial_{y}$.

For a sequence of nonnegative integers $\mathbf{k} = (k_{1}, \cdots, k_{n})$, let $V_{\mathbf{k}} = V_{k_{1}} \otimes \cdots \otimes V_{k_{n}}$, with a natural diagonal $\SL_{2}$-action. Set $2N = |\mathbf{k}|$, for a half integer $N$. There is an isomorphism of $\SL_{2}$-representations $\phi : V_{\mathbf{k}} \to \CC[y_{1}, \cdots, 	y_{n}]_{\le \mathbf{k}}$, where $\CC[y_{1}, \cdots, y_{n}]_{\le \mathbf{k}}$ is the space of polynomials with degree $\le k_{i}$ with respect to $y_{i}$. Then we can take a highest weight vector $v \in V_{\mathbf{k}}$ such that $\phi(v) = 1$.

Let $e_{i}$ (resp. $f_{i}$, $h_{i}$) be the operator on $V_{\mathbf{k}}$ which acts on the $i$-th factor $V_{k_{i}}$ as $e$ (resp. $f$, $h$) and trivially acts on the other factors. On $V_{\mathbf{k}}$, $e = \sum_{i}e_{i}$ and so on. It is straightforward to check that $\phi(f^{\mathbf{j}}v) = \frac{\mathbf{k}!}{(\mathbf{k-j})!}y^{\mathbf{j}}$.

By the action of $h \in \mathfrak{sl}_{2}$, we can decompose $V_{\mathbf{k}}$ into eigenspaces $V_{\mathbf{k}}(\lambda)$ with the eigenvalue $\lambda$. Note that a vector $w \in V_{\mathbf{k}}$ is $\SL_{2}$-invariant if and only if $w \in V_{\mathbf{k}}(0)$ and $e\cdot w = 0$. 

\begin{definition}
Let $\vec{p} = (p_{1}, \cdots, p_{n})$ be a sequence of $n$ distinct points on $\CC \subset \PP^{1}$. Fix an integer $\ell \ge 0$. Let $\mathbf{k} = (k_{1}, \cdots, k_{n})$ be a sequence of nonnegative integers. The space of \textbf{$\mathfrak{sl}_{2}$-conformal blocks of level $\ell$ relative to $\vec{p}$ in $V_{\mathbf{k}}$} is the subspace of $\SL_{2}$-invariants of $V_{\mathbf{k}}$ which is annihilated by the operator $(\sum p_{i}e_{i})^{\ell+1}$. We denote it by $\mathbb{V}_{\ell}(k_{1}, \cdots, k_{n})$.
\end{definition}

\begin{remark}\label{rem:basicpropertiesCB}
\begin{enumerate}
\item Note that there is a natural inclusion 
\[
	\mathbb{V}_{\ell}(k_{1}, \cdots, k_{n}) \subset 
	\mathbb{V}_{\ell+1}(k_{1}, \cdots, k_{n}).
\]
Furthermore, if $\ell \ge N = |\mathbf{k}|/2$, $\mathbb{V}_{\ell}(k_{1}, \cdots, k_{n}) \cong V_{\mathbf{k}}^{\SL_{2}} \cong \mathrm{H}^{0}((\PP^{1})^{n}, \cO(k_{1}, \cdots, k_{n}))^{\SL_{2}}$ because the operator $(\sum p_{i}e_{i})^{N+1}$ is trivial. 
\item For the natural $S_{n}$-action permuting $n$ irreducible factors of $V_{\mathbf{k}}$,
\[
	\mathbb{V}_{\ell}(k_{1}, \cdots, k_{n}) \cong
	\mathbb{V}_{\ell}(k_{\sigma(1)}, \cdots, k_{\sigma(n)})
\]
for every $\sigma \in S_{n}$.
\item If $k_{i} > \ell$ for some $i$, $\mathbb{V}_{\ell}(k_{1}, \cdots, k_{n}) = 0$. 
\item Since $V_{0} \cong \CC$, there is a natural isomorphism 
\[
	\mathbb{V}_{\ell}(k_{1}, \cdots, k_{n}, 0)
	\cong \mathbb{V}_{\ell}(k_{1}, \cdots, k_{n}).
\]
In the physics literature, this isomorphism is called the \textbf{propagation of vacua}.
\end{enumerate}
\end{remark}

The following lemma provides an elementary description of $\mathfrak{sl}_{2}$-conformal blocks. 

\begin{lemma}[\protect{\cite[Lemma 1.3]{Loo09}}]\label{lem:confblock}
An element $\beta \in V_{\mathbf{k}}^{\SL_{2}}$ is in $\mathbb{V}_{\ell}(k_{1}, \cdots, k_{n})$ (relative to $\vec{p} \in \CC^{n} \subset (\PP^{1})^{n}$) if and only if $\phi(\beta)$ has zero of order at least $N - \ell$ at $\vec{p}$.
\end{lemma}

From the identification $V_{k} \cong \CC[y]_{\le k}$ and the description of $\mathfrak{sl}_{2}$-action as differential operators, it is straightforward to see that the map
\[
	V_{\mathbf{k}} \otimes V_{\mathbf{j}} \hookrightarrow 
	V_{\mathbf{k+j}}
\]
given by $f \otimes g \mapsto fg$
induces 
\[
	V_{\mathbf{k}}^{\SL_{2}} \otimes V_{\mathbf{j}}^{\SL_{2}}
	\hookrightarrow V_{\mathbf{k+j}}^{\SL_{2}}.
\]
Furthermore, by the identification of level $\ell$ conformal blocks as polynomials vanishing at $\vec{p}$ with multiplicity $N - \ell$ in Lemma \ref{lem:confblock}, we have the product map on the level of conformal blocks:
\[
	\mathbb{V}_{\ell}(a_{1}, \cdots, a_{n}) \otimes 
	\mathbb{V}_{m}(b_{1}, \cdots, b_{n}) \hookrightarrow
	\mathbb{V}_{\ell+m}(a_{1}+b_{1}, \cdots, a_{n}+b_{n}).
\]

\subsection{Factorization and some combinatorial results on $\mathfrak{sl}_{2}$ conformal blocks}

The rank of $\mathfrak{sl}_{2}$-conformal blocks can be computed by the following inductive formula. 
\begin{proposition}[Fusion rule and factorization rule, \protect{\cite[Section 4]{Bea96a}}]\label{prop:factorization}
Let $k_{1}, \cdots, k_{n}$ be $n$ nonnegative integers such that $k_{i} \le \ell$. 
\begin{enumerate}
\item The rank of $\mathbb{V}_{\ell}(k_{1})$ is one when $k_{1} = 0$. Otherwise the rank is zero.
\item The rank of $\mathbb{V}_{\ell}(k_{1}, k_{2})$ is one when $k_{1} = k_{2}$. Otherwise the rank is zero. 
\item $\rank \mathbb{V}_{\ell}(k_{1}, k_{2}, k_{3}) = 1$ if and only if $\sum k_{i}$ is even, $\sum k_{i}\le 2\ell$ and $k_{j} \le \sum k_{i}/2$. Otherwise the rank is zero.
\item For any $1 \le t \le n$, 
\[
	\rank \mathbb{V}_{\ell}(k_{1}, \cdots, k_{n}) = 
	\sum_{j = 0}^{\ell}
	\left(\rank \mathbb{V}_{\ell}(k_{1}, \cdots, k_{t}, j)\right)
	\left(\rank \mathbb{V}_{\ell}(j, k_{t+1}, \cdots, k_{n})\right).
\]
\end{enumerate}
\end{proposition}

The rank of $\mathfrak{sl}_{2}$-conformal blocks is indeed the number of certain combinatorial objects. Fix a positive integer $\ell$ and let $\mathbf{k} = (k_{1}, \cdots, k_{n})$ be a sequence of integers such that $0 \le k_{i} \le \ell$. 

\begin{definition}[D. Swinarski]\label{def:dblseq}
A \textbf{double sequence} of level $\ell$ and shape $\mathbf{k}$ is a $2 \times n$ matrix
\[
	DS = \left(\begin{array}{cccc}x_{1} & x_{2} & \cdots & x_{n}\\
	y_{1} & y_{2} & \cdots & y_{n}\end{array}\right)
\]
such that:
\begin{enumerate}
\item Each $x_{j}$ and $y_{j}$ is an integer between $0$ and $\ell$;
\item $x_{j} + y_{j} = k_{j}$ for $1 \le j \le n$;
\item For each $1 \le i \le n$,
\[
	x_{i} + \sum_{j=1}^{i-1}(x_{j}-y_{j}) \le \ell;
\]
\item For each $1 \le i \le n$,
\[
	-y_{i} + \sum_{j=1}^{i-1}(x_{j}-y_{j}) \ge 0;
\]
\item $\sum_{j=1}^{n} x_{j} = \sum_{j=1}^{n} y_{j}$.
\end{enumerate}
For a double sequence $DS$, the \textbf{height} $h(DS)$ is the maximum of $x_{i}+\sum_{j=1}^{i-1}(x_{j}-y_{j})$ for $1 \le i \le n-1$. Note that $x_{1} \le h(DS) \le \ell$. 
\end{definition}

\begin{remark}\label{rem:implications}
Definition \ref{def:dblseq} implies several nontrivial implications. First of all, if $\sum_{j=1}^{i-1}(x_{j}-y_{j}) = 0$, then $-y_{i} \ge 0$ by (4). Since $y_{i}$ is nonnegative by (1), $y_{i} = 0$. In particular, $y_{1} = 0$. Also if $\sum_{j=1}^{i}(x_{j}-y_{j}) = 0$, then 
\[
	0 = \sum_{j=1}^{i}(x_{j}- y_{j}) = x_{i} - y_{i} + 
	\sum_{j=1}^{i-1}(x_{j}-y_{j}) \ge x_{i}
\]
by (4). So $x_{i} = 0$ by (1) again. As a special case, $x_{n} = 0$. 
\end{remark}

We can visualize a double sequence by using a \textbf{boxed Catalan path}, introduced by B. Alexeev. For a double sequence $DS$, we can draw a path in the first quadrant of $\RR^{2}$ as the following. Start from the origin. For each $j$, draw $(1,1)$ vector $x_{j}$ times and draw $(1,-1)$ vector $y_{j}$ times. Then (3) and (4) imply that the path is lying on the region $0 \le y \le \ell$. By (5), the ending point is $(|\mathbf{k}|, 0)$. Remark \ref{rem:implications} says that if there is a point $(x, 0)$ at the end of $j$-th move, then the $j$-th move is a downward move and the $(j+1)$-th move is an upward move. The height $h(DS)$ is simply the height of the boxed Catalan path corresponding to $DS$. 

Indeed, (4) implies a stronger condition. For each move, we can draw a lower alternative path for each move by drawing $(1,-1)$ vectors first then drawing $(1,1)$ vectors. See Figure \ref{fig:boxedCatalanpath} for an example. The dashed part can be obtained by drawing $(1,-1)$ vectors first. Then the picture looks like a `boxed' path. Also note that if one of $x_{i}$ or $y_{i}$ is zero, then the corresponding box is simply a line segment. Now item (4) says that each lower corner of a box must have nonnegative height as well. 

\begin{figure}[!ht]
\[
	\left(\begin{array}{cccccc}
	2 & 1 & 1 & 0 & 1 & 0\\
	0 & 1 & 2 & 1 & 0 & 1
	\end{array}\right)
\]
\begin{center}
\begin{tikzpicture}[scale=0.6]
	\draw[line width=0.5pt] (0,0) -- (10, 0);
	\draw[line width=0.5pt] (0,1) -- (10, 1);
	\draw[line width=0.5pt] (0,2) -- (10, 2);
	\draw[line width=0.5pt] (0,3) -- (10, 3);
	\draw[line width=0.5pt] (0,4) -- (10, 4);
	\draw[line width=0.5pt] (0,0) -- (0, 4);
	\draw[line width=0.5pt] (1,0) -- (1, 4);
	\draw[line width=0.5pt] (2,0) -- (2, 4);
	\draw[line width=0.5pt] (3,0) -- (3, 4);
	\draw[line width=0.5pt] (4,0) -- (4, 4);
	\draw[line width=0.5pt] (5,0) -- (5, 4);
	\draw[line width=0.5pt] (6,0) -- (6, 4);
	\draw[line width=0.5pt] (7,0) -- (7, 4);
	\draw[line width=0.5pt] (8,0) -- (8, 4);
	\draw[line width=0.5pt] (9,0) -- (9, 4);
	\draw[line width=0.5pt] (10,0) -- (10, 4);
	\draw[line width=1pt] (0,0) -- (2, 2);
	\draw[line width=1pt] (2,2) -- (3, 3);
	\draw[line width=1pt] (3,3) -- (4, 2);
	\draw[line width=1pt] (4,2) -- (5, 3);
	\draw[line width=1pt] (5,3) -- (7,1);
	\draw[line width=1pt] (7,1) -- (8,0);
	\draw[line width=1pt] (8,0) -- (9,1);
	\draw[line width=1pt] (9,1) -- (10,0);
	\draw[dashed, line width=1pt] (2,2) -- (3,1);
	\draw[dashed, line width=1pt] (3,1) -- (4,2);
	\draw[dashed, line width=1pt] (4,2) -- (6,0);
	\draw[dashed, line width=1pt] (6,0) -- (7,1);
	\fill (2,2) circle (4pt);
	\fill (4,2) circle (4pt);
	\fill (7,1) circle (4pt);
	\fill (8,0) circle (4pt);
	\fill (9,1) circle (4pt);
	\fill (10,0) circle (4pt);
\end{tikzpicture}
\end{center}
\caption{A double sequence of level $\ge 3$ and shape $(2,2,3,1,1,1)$ and the corresponding boxed Catalan path}
\label{fig:boxedCatalanpath}
\end{figure}
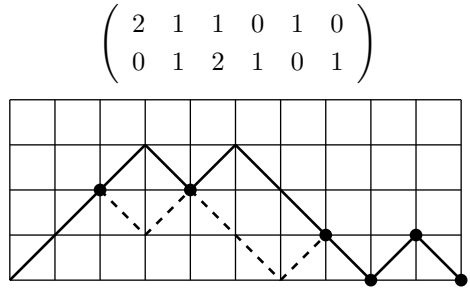

Let $S(\ell, \mathbf{k})$ be the set of double sequences of level $\ell$ and shape $\mathbf{k}$. We have learned the following result of B. Alexeev from D. Swinarski. 

\begin{proposition}[B. Alexeev]\label{prop:CBandboxCatpath}
\[
	\rank \mathbb{V}_{\ell}(k_{1}, \cdots, k_{n}) = 
	|S(\ell, \mathbf{k})|.
\]
\end{proposition}

\begin{proof}
It is straightforward to check the proposition for $n = 1, 2, 3$. Also the number of double sequences (or equivalently, boxed Catalan paths) satisfies the factorization rule. Indeed, consider a double sequence 
\[
	DS = \left(\begin{array}{cccc}
	x_{1} & x_{2} & \cdots & x_{n}\\
	y_{1} & y_{2} & \cdots & y_{n}\end{array}\right)
\]
and the corresponding boxed Catalan path. For any $1 \le t \le n$, after $t$-th move, the $y$-coordinate is one of $0, 1, \cdots, \ell$. If the coordinate is $h$, we can construct two double sequences 
\[
	DS' = \left(\begin{array}{ccccc}
	x_{1} & x_{2} & \cdots & x_{t} & 0\\
	y_{1} & y_{2} & \cdots & y_{t} & h\end{array}\right), \quad
	DS'' = \left(\begin{array}{cccc}
	h & x_{t+1} & \cdots & x_{n}\\
	0 & y_{t+1} & \cdots & y_{n}\end{array}\right).
\]
It is straightforward to check that $DS' \in S(\ell, (k_{1}, \cdots, k_{t}, h))$ and $DS'' \in S(\ell, (h, k_{t+1}, \cdots, k_{n}))$. Conversely, for two double sequences $DS' \in S(\ell, (k_{1}, \cdots, k_{t}, h))$ and $DS'' \in S(\ell, (h, k_{t+1}, \cdots, k_{n}))$, by removing the last column (resp. the first column) of $DS'$ (resp. $DS''$) and merging them, we obtain a double sequence in $S(\ell, \mathbf{k})$. Thus we have 
\[
	|S(\ell, \mathbf{k})| = \sum_{h=1}^{\ell}|S(\ell, (k_{1}, \cdots, k_{t}, h))|
	\cdot |S(\ell, (h, k_{t+1}, \cdots, k_{n}))|.
\]
\end{proof}

\section{The effective cone of the moduli space of parabolic vector bundles}\label{sec:effectiveconegeneral}

In this section, we compute the effective cone of $\cM(\vec{a})$ with an arbitrary general parabolic weight $\vec{a}$. In this section, we assume that the number $n$ of parabolic points is at least 5. The following is a direct consequence of Proposition \ref{prop:generalmoduli}. 

\begin{lemma}\label{lem:generaleffectivecone}
Let $\vec{a} = (a_{1}, \cdots, a_{n})$ be a general parabolic weight such that $\cM(\vec{a})$ has the maximal Picard number $n+1$. Then $\mathrm{rank}\;\Pic(\cM(\vec{a}))_{\QQ} = n+1$ and $\mathrm{Eff}(\cM(\vec{a}))$ is identified with $\mathrm{Eff}(\mathrm{Bl}_{[\vec{p}]}(\PP^{1})^{n}\git_{L}\SL_{2})$ for $L$ with a maximal stable locus.
\end{lemma}

Thus for $\cM(\vec{a})$ with Picard number $n+1$, to compute $\mathrm{Eff}(\cM(\vec{a}))$, it suffices to compute $\mathrm{Eff}(\cM^{+})$, where $\cM^{+} := \mathrm{Bl}_{[\vec{p}]}(\PP^{1})^{n}\git_{L}\SL_{2}$ with $L = \cO(a_{1}, \cdots, a_{n})$. Let $E$ be the exceptional divisor of the blow-up $\cM^{+} \to (\PP^{1})^{n}\git_{L}\SL_{2}$. Since $\Pic(\cM^{+})_{\QQ}$ is generated by $\cO(D_{\{i,j\}}) = \overline{\cO}(e_{i}+e_{j})$ and $E$, we can uniquely write a $\QQ$-line bundle on $\cM^{+}$ (so on $\cM(\vec{a})$) as
\[
	\overline{\cO}(b_{1}, \cdots, b_{n}) - tE
\]
for some $b_{i}$ and $t$. 

The main result of this section is the following theorem. 

\begin{theorem}\label{thm:effectiveconegeneral}
Let $\vec{a}$ be a general parabolic weight such that $\cM(\vec{a})$ has the maximal Picard number $n+1$. Then the effective cone $\mathrm{Eff}(\cM(\vec{a}))$ is polyhedral and generated by $\overline{\cO}(\sum_{j \in I}e_{j}) - (i-1)E$ for every $I \subset [n]$ with $|I| = 2i$ for $0 \le i \le \lfloor n/2\rfloor$. All $\overline{\cO}(\sum_{j \in I}e_{j}) - (i-1)E$ are extremal, thus there are precisely $2^{n-1}$ extremal rays. 
\end{theorem}

We give the proof of Theorem \ref{thm:effectiveconegeneral} after discussing several lemmas. 

The following observations are simple but important to us.

\begin{lemma}\label{lem:effectivedivisorasCB}
For any $t \ge 0$, the linear system $|\overline{\cO}(b_{1}, \cdots, b_{n}) - tE|$ is naturally identified with $\mathbb{V}_{N-t}(b_{1}, \cdots, b_{n})$ where $N = (\sum b_{i})/2$. 
\end{lemma}

\begin{proof}
Since $\cM^{+} \to (\PP^{1})^{n}\git_{L}\SL_{2}$ is a blow-up at a smooth point, $|\overline{\cO}(b_{1}, \cdots, b_{n}) - tE|$ is the sub linear system of $|\overline{\cO}(b_{1}, \cdots, b_{n})|$ consisting of the sections vanishing at $[\vec{p}]$ with multiplicity $\ge t$. By Lemma \ref{lem:confblock}, it is identified with $\mathbb{V}_{N-t}(b_{1}, \cdots, b_{n})$.
\end{proof}

\begin{lemma}
For any $I \subset [n]$ with $|I| = 2i$ and $0 \le i \le \lfloor n/2\rfloor$, the linear system $|\overline{\cO}(\sum_{j \in I}e_{j}) - (i-1)E|$ is nonempty. It is a conformal block of level one. 
\end{lemma}

\begin{proof}
For $i \ge 1$, note that $|\overline{\cO}(\sum_{j \in I}e_{j}) - (i-1)E| = \mathbb{V}_{1}(b_{1}, \cdots, b_{n})$ where $b_{j} = 1$ if $j \in I$ and $b_{j} = 0$ otherwise. By  Proposition \ref{prop:factorization}, we have $\rank \mathbb{V}_{1}(b_{1}, \cdots, b_{n}) = 1$. When $i = 0$, we have $|\overline{\cO}+E| = |E| \ne \emptyset$, which may be formally identified with $\mathbb{V}_{1}(0,0, \cdots, 0)$.
\end{proof}

The next lemma is a key combinatorial result for the computation of the effective cone. 

\begin{lemma}\label{lem:modifyingdoubleseq}
Let $DS$ be a double sequence of level $\ell$, height $h(DS) > 1$, and of shape $\mathbf{k} = (k_{1}, \cdots, k_{n})$ with $k_{1} \ge \cdots \ge k_{n} > 0$. There is a nonempty even subset $T \subset [n]$ such that there is a double sequence $DS'$ with level $\ell-1$, height $h(DS') = h(DS)-1$, and shape $\mathbf{k}' = (k_{1}', \cdots, k_{n}')$ where 
\[
	k_{j}' = \begin{cases}k_{j}, & j \notin T,\\
	k_{j}-1, & j \in T. \end{cases}
\]
\end{lemma}

\begin{proof}
We will construct a new double sequence 
\[
	DS' = \left(\begin{array}{cccc}x_{1}' & x_{2}' & \cdots & x_{n}'\\
	y_{1}' & y_{2}' & \cdots & y_{n}'\end{array}\right)
\]
and $T$ as the following. At the beginning, set $DS' = DS$ and $T = \emptyset$. The reader can understand the modification below by identifying $DS$ with the corresponding boxed Catalan path. 

First of all, consider the special case that the boxed Catalan path meets the $x$-axis at the starting point and the end point only. Set $c_{1} := 1$ and put $c_{1}$ in $T$. Let $c_{2} > c_{1}$ be the smallest index such that 1) $-y_{c_{2}} + \sum_{i=1}^{c_{2}-1} (x_{i}-y_{i}) = 0$ and 2) between $c_{1}$-th move and $c_{2}$-th move, the path reaches the height $h(DS)$ at least once. If $y_{c_{2}} = 0$, then $\sum_{i=1}^{c_{2}-1}(x_{i}-y_{i}) = 0$ and the path meets $x$-axis after $c_{2}-1$-th move, too. From the assumption, we always have $y_{c_{2}} > 0$. Note that $n$ satisfies all of two conditions, so the set of indices satisfying them is nonempty. Put $c_{2}$ in $T$. 

We will continue the construction of $T$ as the following. If the path reaches the height $h(DS)$ after $c_{2}$-th move, then let $c_{2j+1} > c_{2j}$ be the smallest index so that $x_{c_{2j+1}} > 0$. There must be at least one such index because if not, then after $c_{2j}$-th move, it never reach the height $h(DS)$. Let $c_{2j+2} > c_{2j+1}$ be the smallest index satisfying two conditions in the previous paragraph. Then $y_{c_{2j+2}} > 0$. Set $c_{2j+1}, c_{2j+2} \in T$. Continue this procedure until there is no remaining intersection with the path and $y = h(DS)$. After that, set $x_{c_{2j+1}}' := x_{c_{2j+1}}-1$ and $y_{c_{2j+2}}' := y_{c_{2j+2}}-1$ for all $j \ge 0$. 

By the construction, $T$ is even and nonempty, $h(DS') = h(DS) - 1$, the new level is at most $\ell - 1$, and $\mathbf{k}'$ is that in the statement of the lemma. Conditions 1, 2, and 5 in Definition \ref{def:dblseq} are clear from the construction. Also it is a routine computation to verify conditions 3 and 4. Thus $DS'$ is a double sequence. See Figure \ref{fig:modification} for an example of the modification. 

In general, if there are several points on which the path and the $x$-axis intersect, then we can modify each part over the $x$-axis separately by using the above method. It is clear that we obtain a new double sequence satisfying the assumption. 
\end{proof}

\begin{figure}[!ht]
\begin{center}
\begin{tikzpicture}[scale=0.4]
	\draw[line width=0.5pt] (0,0) -- (26, 0);
	\draw[line width=0.5pt] (0,1) -- (26, 1);
	\draw[line width=0.5pt] (0,2) -- (26, 2);
	\draw[line width=0.5pt] (0,3) -- (26, 3);
	\draw[line width=0.5pt] (0,4) -- (26, 4);
	\draw[line width=0.5pt] (0,5) -- (26, 5);
	\draw[line width=0.5pt] (0,0) -- (0, 5);
	\draw[line width=0.5pt] (1,0) -- (1, 5);
	\draw[line width=0.5pt] (2,0) -- (2, 5);
	\draw[line width=0.5pt] (3,0) -- (3, 5);
	\draw[line width=0.5pt] (4,0) -- (4, 5);
	\draw[line width=0.5pt] (5,0) -- (5, 5);
	\draw[line width=0.5pt] (6,0) -- (6, 5);
	\draw[line width=0.5pt] (7,0) -- (7, 5);
	\draw[line width=0.5pt] (8,0) -- (8, 5);
	\draw[line width=0.5pt] (9,0) -- (9, 5);
	\draw[line width=0.5pt] (10,0) -- (10, 5);
	\draw[line width=0.5pt] (11,0) -- (11, 5);
	\draw[line width=0.5pt] (12,0) -- (12, 5);
	\draw[line width=0.5pt] (13,0) -- (13, 5);
	\draw[line width=0.5pt] (14,0) -- (14, 5);
	\draw[line width=0.5pt] (15,0) -- (15, 5);
	\draw[line width=0.5pt] (16,0) -- (16, 5);
	\draw[line width=0.5pt] (17,0) -- (17, 5);
	\draw[line width=0.5pt] (18,0) -- (18, 5);
	\draw[line width=0.5pt] (19,0) -- (19, 5);
	\draw[line width=0.5pt] (20,0) -- (20, 5);
	\draw[line width=0.5pt] (21,0) -- (21, 5);
	\draw[line width=0.5pt] (22,0) -- (22, 5);
	\draw[line width=0.5pt] (23,0) -- (23, 5);
	\draw[line width=0.5pt] (24,0) -- (24, 5);
	\draw[line width=0.5pt] (25,0) -- (25, 5);
	\draw[line width=0.5pt] (26,0) -- (26, 5);
	\draw[line width=1pt] (0,0) -- (5,5);
	\draw[line width=1pt] (5,5) -- (9,1);
	\draw[line width=1pt] (9,1) -- (12,4);
	\draw[line width=1pt] (12,4) -- (15,1);
	\draw[line width=1pt] (15,1) -- (19,5);
	\draw[line width=1pt] (19,5) -- (23,1);
	\draw[line width=1pt] (23,1) -- (24,2);
	\draw[line width=1pt] (24,2) -- (26,0);
	\fill (5,5) circle (4pt);
	\fill (9,1) circle (4pt);
	\fill (13,3) circle (4pt);
	\fill (15,1) circle (4pt);
	\fill (17,3) circle (4pt);
	\fill (19,5) circle (4pt);
	\fill (21,3) circle (4pt);
	\fill (23,1) circle (4pt);
	\fill (25,1) circle (4pt);
	\fill (26,0) circle (4pt);

	\node at (2.5, 3.5) {$c_{1}$};
	\node at (10.5, 3.5) {$c_{2}$};
	\node at (15.5, 2.5) {$c_{3}$};
	\node at (24.5, 2.5) {$c_{4}$};

	\node at (13, -1) {$\Downarrow$};

	\draw[line width=0.5pt] (0,-7) -- (26, -7);
	\draw[line width=0.5pt] (0,-6) -- (26, -6);
	\draw[line width=0.5pt] (0,-5) -- (26, -5);
	\draw[line width=0.5pt] (0,-4) -- (26, -4);
	\draw[line width=0.5pt] (0,-3) -- (26, -3);
	\draw[line width=0.5pt] (0,-2) -- (26, -2);

	\draw[line width=0.5pt] (0,-7) -- (0, -2);
	\draw[line width=0.5pt] (1,-7) -- (1, -2);
	\draw[line width=0.5pt] (2,-7) -- (2, -2);
	\draw[line width=0.5pt] (3,-7) -- (3, -2);
	\draw[line width=0.5pt] (4,-7) -- (4, -2);
	\draw[line width=0.5pt] (5,-7) -- (5, -2);
	\draw[line width=0.5pt] (6,-7) -- (6, -2);
	\draw[line width=0.5pt] (7,-7) -- (7, -2);
	\draw[line width=0.5pt] (8,-7) -- (8, -2);
	\draw[line width=0.5pt] (9,-7) -- (9, -2);
	\draw[line width=0.5pt] (10,-7) -- (10, -2);
	\draw[line width=0.5pt] (11,-7) -- (11, -2);
	\draw[line width=0.5pt] (12,-7) -- (12, -2);
	\draw[line width=0.5pt] (13,-7) -- (13, -2);
	\draw[line width=0.5pt] (14,-7) -- (14, -2);
	\draw[line width=0.5pt] (15,-7) -- (15, -2);
	\draw[line width=0.5pt] (16,-7) -- (16, -2);
	\draw[line width=0.5pt] (17,-7) -- (17, -2);
	\draw[line width=0.5pt] (18,-7) -- (18, -2);
	\draw[line width=0.5pt] (19,-7) -- (19, -2);
	\draw[line width=0.5pt] (20,-7) -- (20, -2);
	\draw[line width=0.5pt] (21,-7) -- (21, -2);
	\draw[line width=0.5pt] (22,-7) -- (22, -2);
	\draw[line width=0.5pt] (23,-7) -- (23, -2);
	\draw[line width=0.5pt] (24,-7) -- (24, -2);
	\draw[line width=0.5pt] (25,-7) -- (25, -2);
	\draw[line width=0.5pt] (26,-7) -- (26, -2);	
	\draw[line width=1pt] (0,-7) -- (4,-3);
	\draw[line width=1pt] (4,-3) -- (8,-7);
	\draw[line width=1pt] (8,-7) -- (11,-4);
	\draw[line width=1pt] (11,-4) -- (13,-6);
	\draw[line width=1pt] (13,-6) -- (16,-3);
	\draw[line width=1pt] (16,-3) -- (20,-7);
	\draw[line width=1pt] (20,-7) -- (21,-6);
	\draw[line width=1pt] (21,-6) -- (22,-7);
	\fill (4,-3) circle (4pt);
	\fill (8,-7) circle (4pt);
	\fill (11,-4) circle (4pt);
	\fill (13,-6) circle (4pt);
	\fill (14,-5) circle (4pt);
	\fill (16,-3) circle (4pt);
	\fill (18,-5) circle (4pt);
	\fill (20,-7) circle (4pt);
	\fill (21,-6) circle (4pt);
	\fill (22,-7) circle (4pt);
\end{tikzpicture}
\end{center}
\caption{An example of the modification of a boxed Catalan path}
\label{fig:modification}
\end{figure}
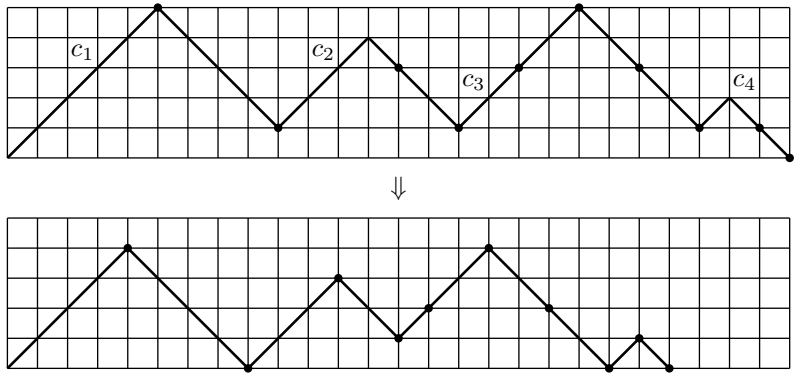

\begin{proof}[Proof of Theorem \ref{thm:effectiveconegeneral}]
We may assume that $\cM(\vec{a}) = \cM^{+}$, that is, a blow-up of $(\PP^{1})^{n}\git_{L}\SL_{2}$ at $[\vec{p}]$ where $L = \cO(a_{1}, \cdots, a_{n})$.

Take an effective divisor $D \in |\overline{\cO}(k_{1}, \cdots, k_{n}) - tE|$. Since $\cM^{+} \to (\PP^{1})^{n}\git_{L}\SL_{2}$ is a blow-up at a point and $E$ is the exceptional divisor, the intersection of $\mathrm{Eff}(\cM^{+})$ with the half space $t \le 0$ is generated by the extremal rays of $(\PP^{1})^{n}\git_{L}\SL_{2}$ and $E$. 

Suppose that $t \ge 0$. Then $D$ is the zero set of a section $s \in \mathbb{V}_{N-t}(k_{1}, \cdots, k_{n})$ where $2N = \sum k_{i}$. By rearranging the indices, we may assume that $k_{1} \ge \cdots \ge k_{n}$. Also by the propagation of vacua, we may assume that $k_{n} > 0$. Since $\mathbb{V}_{N-t}(k_{1}, \cdots, k_{n}) \ne 0$, by Proposition \ref{prop:CBandboxCatpath}, there is a double sequence $DS$ of level $N-t$ and shape $\mathbf{k}$. 

By Lemma \ref{lem:modifyingdoubleseq}, we can construct a set $T \subset [n]$ and a double sequence $DS'$ of level $N-t-1$ and shape $\mathbf{k}'$ (see Lemma \ref{lem:modifyingdoubleseq} for notations). By Proposition \ref{prop:CBandboxCatpath} again, $\rank\mathbb{V}_{N-t-1}(k_{1}', \cdots, k_{n}') > 0$. Furthermore, if we set
\[
	b_{i} = \begin{cases} 1, & i \in T,\\ 0, & i \notin T,\end{cases}
\]
then $\mathbb{V}_{1}(b_{1}, \cdots, b_{n}) \ne 0$ by the factorization rule and we have a morphism
\[
	\mathbb{V}_{1}(b_{1}, \cdots, b_{n}) \otimes 
	\mathbb{V}_{N-t-1}(k_{1}', \cdots, k_{n}')
	\hookrightarrow \mathbb{V}_{N-t}(k_{1}, \cdots, k_{n}), 
\]
which is given by the multiplication of sections (see Section \ref{ssec:quickdef}).
If we set $|\mathbf{k'}| = \sum k_{i}'$, $N' = |\mathbf{k'}|/2$, then $N - t - 1 = N' - (t+1 - |T|/2)$. Therefore the divisor $D$ is numerically equivalent to the sum of a divisor corresponding to a level one conformal block and an effective divisor in $|\overline{\cO}(k_{1}', \cdots, k_{n}') - t'E|$ where $t' := t+1-|T|/2 \le t$. 

By induction on $|\mathbf{k}|$, we can see that $D$ is numerically equivalent to an effective sum of level one conformal blocks, $E$ and a divisor corresponding to $\mathbb{V}_{r}(c_{1}, \cdots, c_{n})$ where $c_{i}$ is either 0 or 1 and $r \ge 1$. The very last divisor is an effective sum of a level one conformal block and the divisor $E$. In summary, $D$ is in the cone generated by level one conformal blocks and $E$. 

It remains to show that all of the generators are indeed extremal rays. It is shown in Proposition \ref{prop:extremal} below.
\end{proof}

\begin{proposition}\label{prop:extremal}
For $0 \le i \le \lfloor n/2 \rfloor$ and $I \subset [n]$ with $|I| = 2i$, a divisor in $|\overline{\cO}(\sum_{j \in I}e_{j}) - (i-1)E|$ is an extremal ray of $\mathrm{Eff}(\cM(\vec{a}))$. 
\end{proposition}

\begin{lemma}\label{lem:indepset}
For each $n \ge 3$, there are $n$ subsets $J_{1}, \cdots, J_{n} \subset [n]$ such that 
\begin{enumerate}
\item $|J_{k}| = n-2$;
\item $\{\sum_{j \in J_{k}}e_{j}\}$ form a basis of $\QQ^{n}$.
\end{enumerate}
\end{lemma}

\begin{proof}
For $1 \le k \le n-1$, let $J_{k} = [n-1] - \{k\}$. Let $A$ be an $(n-1) \times (n-1)$ matrix whose $k$-th row is $\sum_{j \in J_{k}}e_{j}$. If we denote the square matrix whose all entries are one by $J$ and if $I$ is the identity matrix then $A = J - I$. It is straightforward to check that the characteristic polynomial of $J$ is $P(t) = (-1)^{n-1}t^{n-2}(t-n+1)$. Now $\det A = \det (J - I) = P(1) \ne 0$, thus $\sum_{j \in J_{1}}e_{j}, \cdots, \sum_{j \in J_{n-1}}e_{j}$ are linearly independent. Finally, take any $J' \subset [n-1]$ where $|J'| = n-3$ and define $J_{n} = J'\cup \{n\}$. Then $\sum_{j \in J_{1}}e_{j}, \cdots, \sum_{j \in J_{n}}e_{j}$ are linearly independent. 
\end{proof}

\begin{proof}[Proof of Proposition \ref{prop:extremal}]
Let 
\[
	S = \{\overline{\cO}(\sum_{j \in I}e_{j}) - (i-1)E\;|\; 
	0 \le i \le \lfloor n/2\rfloor, I \subset [n], |I| = 2i\},
\]
the set of generators of $\mathrm{Eff}(\cM(\vec{a}))$. For $i = 0$, $E$ is the exceptional divisor of a blow-up, so it is extremal. For each $i \ge 1$ and $I \subset [n]$, we will construct $n$ linearly independent functionals $\ell_{1}, \cdots, \ell_{n} \in \Pic(\cM(\vec{a}))_{\QQ}^{*}$ such that
\begin{enumerate}
\item $\ell_{k}(\overline{\cO}(\sum_{j \in I}e_{j}) - (i-1)E) = 0$;
\item $\ell_{k}(D) \ge 0$ for all $D \in S$;
\end{enumerate}
for $1 \le k \le n$. Since $\Pic(\cM(\vec{a}))_{\QQ}$ has rank $n+1$, we can conclude that all elements of $S$ are extremal rays of $\Pic(\cM(\vec{a}))_{\QQ}$. By symmetry, it is enough to show for $I = \{1, 2, \cdots, 2i\}$, i.e., $\overline{\cO}(\sum_{j=1}^{2i}e_{j}) - (i-1)E$. When  $i \ge 2$, let $J_{1}, \cdots, J_{2i}$ be $2i$ subsets of $[2i]$ constructed in Lemma \ref{lem:indepset}. Define $\ell_{k}$ as:
\begin{enumerate}
\item $\ell_{k}(\overline{\cO}(\sum a_{j}e_{j})-tE) = \sum_{j \in J_{k}}a_{j} + \sum_{j > 2i}a_{j} - 2t$ for $1 \le k \le 2i$;
\item $\ell_{k}(\overline{\cO}(\sum a_{j}e_{j})-tE) = a_{k}$ for $2i < k \le n$.
\end{enumerate}
If $i = 1$ (thus $I = \{1, 2\}$), define $\ell_{k}$ as:
\begin{enumerate}
\item $\ell_{k}(\overline{\cO}(\sum a_{j}e_{j})-tE) = \sum_{j = 1}^{n}a_{j} - a_{k} - t$ for $k = 1, 2$;
\item $\ell_{k}(\overline{\cO}(\sum a_{j}e_{j})-tE) = a_{k}$ for $2 < k \le n$.
\end{enumerate}

It is a routine computation to check that those linear functionals are linearly independent, 
\[
	\ell_{k}(\overline{\cO}(\sum_{j=1}^{2i}e_{j}) - (i-1)E) = 0,
\]
and $\ell_{k}(D) \ge 0$ for every $D \in S$.
\end{proof}

\begin{remark}
If $\cM(\vec{a})$ is general (i.e. stability coincides with semistability) but does not have the maximal Picard number, then it is a rational contraction (a composition of finitely many flips and divisorial contractions) of $\cM(\vec{a}')$ with Picard number $n+1$. If we denote the rational contraction by $\phi : \cM(\vec{a}') \dashrightarrow \cM(\vec{a})$, then there is a well-defined push-forward 
\[
	\phi_{*} : \mathrm{Pic}(\cM(\vec{a}'))_{\QQ} \to 
	\mathrm{Pic}(\cM(\vec{a}))_{\QQ}
\]
and $\mathrm{Eff}(\cM(\vec{a})) = \im \; \phi_{*}(\mathrm{Eff}(\cM(\vec{a}')))$ since all divisors on $\cM(\vec{a})$ are Cartier. So $\mathrm{Eff}(\cM(\vec{a}))$ is generated by $\{\phi_{*}(\overline{\cO}(\sum_{j \in I}e_{j}) - (i-1)E)\}$. Therefore essentially Theorem \ref{thm:effectiveconegeneral} gives $\mathrm{Eff}(\cM(\vec{a}))$ for a general parabolic weight $\vec{a}$. 
\end{remark}

\section{Theta divisors and birational models}\label{sec:thetadivisor}

Theorem \ref{thm:effectiveconegeneral} tells us that any effective divisor on $\cM(\vec{a})$ can be described as a nonnegative linear combination of conformal blocks and the exceptional divisor $E$. This result has an interesting consequence (Theorem \ref{thm:Moriprogram}). 

\begin{lemma}\label{lem:MDS}
For a general parabolic weight $\vec{a} \in W^0$, $\cM(\vec{a})$ is a Mori dream space.
\end{lemma}

\begin{proof}
Abe showed that when $\vec{b} = (1/2, \cdots, 1/2)$, $\cM(\vec{b})$ is a Fano variety (\cite[Proposition 2.7]{Abe04}). It is straightforward to see that $\vec{b}$ is on a stability wall only if $n$ is even. Thus by \cite[Corollary 1.3.2]{BCHM10}, $\cM(\vec{b})$ is a Mori dream space if $n$ is odd. Set $\cM(\vec{b}^{\epsilon}) := \cM(\vec{b})$. 

When $n$ is even, $\vec{b}$ lies on a stability wall so in this case the Picard number of $\cM(\vec{b})$ is not maximal. But if we perturb the parabolic weight slightly, then the anticanonical divisor is on the boundary of the nef cone and if we subtract a boundary divisor with small coefficient, then it becomes ample. Thus for the perturbed parabolic weight $\vec{b}^{\epsilon}$, $\cM(\vec{b}^{\epsilon})$ has the maximal Picard number and it is log Fano. By \cite[Corollary 1.3.2]{BCHM10} again, $\cM(\vec{b}^{\epsilon})$ is a Mori dream space, too. 

Therefore in any case, $\cM(\vec{b}^{\epsilon})$ is a Mori dream space and has the maximal Picard number. Because $\cM(\vec{b}^{\epsilon})$ and $\cM(\vec{a})$ are connected by finitely many flips, if one is a Mori dream space then so is the other. 

Finally, for a general parabolic weight $\vec{a}$, the space $\cM(\vec{a})$ is a smooth contraction of certain $\cM(\vec{a}')$ with the maximal Picard number. Thus it is a Mori dream space, too.
\end{proof}

By above lemma and \cite[Proposition 1.11]{HK00}, we know that for any effective divisor we can construct a projective model
\[
	\cM(\vec{a})(D) := \proj \bigoplus_{m \ge 0}
	\mathrm{H}^{0}(\cM(\vec{a}),\lfloor\cO(mD)\rfloor)
\]
and there are only finitely many of them. 

In \cite{Pau96}, Pauly described a generalization of the theta divisor on the Jacobian of a curve, to the moduli space of parabolic vector bundles. 

\begin{definition}[\protect{\cite[Theorem 3.3]{Pau96}}]
In $\Pic(\cM(\vec{a}))_{\QQ}$, the \textbf{theta divisor} $\Theta_{\vec{a}}$ is a divisor such that for any family $(\cE, \{\cV_{i}\})$ over $\pi : S \to \cM(\vec{a})$, 
\[
	\pi^{*}(\Theta_{\vec{a}}) = (\det R\pi_{!}\cE)^{-k} \otimes 
	(\bigotimes_{i=1}^{n}\det \cQ_{i})^{ka_{i}} 
	\otimes (\det \cE|_{S \times \{y\}})^{e}
\]
where $k$ is the smallest positive integer such that $ka_{i}$ are all integers, $y$ is a point of $\PP^{1}$ and $e$ is determined by $e = k(1-(\sum a_{i})/2)$.
\end{definition}

Pauly showed that $\Theta_{\vec{a}}$ is ample (\cite[Theorem 3.3]{Pau96}) and 
\[
	\mathrm{H}^{0}(\cM(\vec{a}), \Theta_{\vec{a}}) \cong 
	\mathbb{V}_{k}(ka_{1}, \cdots, ka_{n})
\]
(\cite[Corollary 6.7]{Pau96}) when $0 < a_{i} < 1$, or equivalently, $0 < ka_{i} < k$ for every $1 \le i \le n$.

Now we can prove the second main theorem of this paper. 

\begin{theorem}\label{thm:Moriprogram}
For any $\QQ$-divisor $D \in \mathrm{int}\;\mathrm{Eff}(\cM(\vec{a}))$, the birational model $\cM(\vec{a})(D)$ is isomorphic to $\cM(\vec{b})$ for some parabolic weight $\vec{b}$. 
\end{theorem}

\begin{proof}
When $\mathrm{rank}\;\mathrm{Pic}(\cM(\vec{a}))_{\QQ}$ is not $n+1$, $\cM(\vec{a})$ is a rational contraction of $\cM(\vec{a}')$ with the maximal Picard number. Then $\mathrm{Eff}(\cM(\vec{a}))$ is embedded into $\mathrm{Eff}(\cM(\vec{a}'))$ naturally. So it suffices to show for $\cM(\vec{a})$ with the maximal Picard number. Write $D$ as $\overline{\cO}(b_{1}, \cdots, b_{n}) - tE$. 

First of all, suppose that $t > 0$. We may replace $D$ by its integral multiple and assume that $D$ is sufficiently divisible integral divisor. Then $|D| = \mathbb{V}_{N-t}(b_{1}, \cdots, b_{n})$ where $N = \sum b_{i}/2$. If $D$ is in the interior of the effective cone, $mD - E$ is effective for $m \gg 0$. This implies that $\mathbb{V}_{m(N-t)-1}(mb_{1}, \cdots , mb_{n}) \ne 0$, so $mb_{i} \le m(N-t)-1$ by Item (3) of Remark \ref{rem:basicpropertiesCB}. Therefore $b_{i} < N-t$. Then $D$ is a theta divisor on $\cM(1/(N-t)\vec{b})$. Since $D$ is ample on $\cM(1/(N-t)\vec{b})$, $\cM(\vec{a})(D) \cong \cM(1/(N-t)\vec{b})$. 

If $t \le 0$, define $D' := \overline{\cO}(b_{1}, \cdots, b_{n}) = D+tE$. Then $\cM(\vec{a})(D') \cong (\PP^{1})^{n}\git_{L}\SL_{2}$, where $L = \cO(b_{1}, \cdots, b_{n})$. Since $E$ is the exceptional divisor of the rational contraction $\cM(\vec{a}) \dashrightarrow (\PP^{1})^{n}\git_{L}\SL_{2}$, $\cM(\vec{a})(D) = \cM(\vec{a})(D')$. By Proposition \ref{prop:parabolicGIT}, $(\PP^{1})^{n}\git_{L}\SL_{2} \cong \cM(c\vec{b})$ for $0 < c < 2/(\sum b_{i})$. 
\end{proof}

\begin{remark}\label{rem:boundary}
The projective models for $D \in \partial \mathrm{Eff}(\cM(\vec{a}))$ are also described by moduli spaces of parabolic vector bundles. 

There are two different types of degenerations of a theta divisor $\mathbb{V}_{k}(k_{1}, \cdots, k_{n})$. One is the case that $k_{i} = 0$ for some $i$. The other one is that $k_{i} = k$ (Of course, both cases can arise together). If $k_{1}, \cdots, k_{r} > 0$ and $k_{r+1} = \cdots = k_{n} = 0$, then from the propogation of vacua, $\mathbb{V}_{k}(k_{1}, \cdots, k_{n}) \cong \mathbb{V}_{k}(k_{1}, \cdots, k_{r})$, thus the projective model is a moduli space of parabolic vector bundles with fewer parabolic points. 

The projective model corresponding to the second degeneration is described by Bertram in \cite[Section 3]{Ber94}. We may assume that $k_{1} = \cdots = k_{r} = k$ and $k_{i} < k$ for $i > r$. For a family of parabolic bundles $(\cE, \{\cV_{i}\}, \vec{a})$ of degree $d$ over $S$, we can construct a new family of parabolic bundles of degree $d-r$ and $(n-r)$ marked points as taking the kernel of 
\[
	\cE \to \bigoplus_{i=1}^{r}\cE|_{p_{i}}/\cV_{i},
\]
and taking $n-r$ subspaces $\cV_{i}$ for $r < i \le n$. 

Thus we have a rational map 
\[
	p : \cM(\vec{a}, d) \dashrightarrow \cM(\vec{a}', d-r),
\]
where $\vec{a}' = (a_{r+1}, \cdots, a_{n})$. In general, $p$ is not regular because it does not guarantee the stability of the induced family. But when $a_{i} \to 1$ for $1 \le i \le r$ (equivalently, $k_{i}$ is very close to $k$ for every $1 \le i \le r$), $p$ is a regular morphism. Bertram showed that the pull-back of the canonical polarization from GIT on $\cM(\vec{a}', -r)$ to $\cM(\vec{a})$ is precisely $\mathbb{V}_{k}(k, \cdots, k, k_{r+1}, \cdots, k_{n})$, where $\vec{a} = \frac{1}{k}(k_{1}, \cdots, k_{n})$ and $k_{i}$ is very close to $k$ for every $1 \le i \le r$. 
\end{remark}

\acknowledgement

We would like to thank Young-Hoon Kiem and David Swinarski, and anonymous referees for many invaluable suggestions for earlier drafts of this paper. The first author thanks Ana-Maria Castravet, Christopher Manon, and Swarnava Mukhopadhyay for indicating some references.

\end{document}